\documentclass[a4paper,12pt,reqno]{amsart}
\usepackage{amssymb,amsfonts,amsmath,amscd,amsthm,latexsym,euscript}%,textcomp,amscd,
\usepackage{times}
\usepackage{graphicx}
\usepackage{tikz}
\usetikzlibrary{cd}
\usetikzlibrary{arrows.meta}
\usepackage{geometry}   
\usepackage{xcolor}
\usepackage{float}
\usepackage{algorithm}
\usepackage{algpseudocode}

\usepackage[bookmarks=false]{hyperref}
\hypersetup{%
    colorlinks=true,        % false: boxed links; true: colored links
    linkcolor=blue,          % color of internal links (change box color with linkbordercolor)
    citecolor=blue,         % color of links to bibliography
    urlcolor=blue           % color of external links
    }

\usepackage{seqsplit}

\newcommand{\dashedhrule}{%
  \noindent\tikz{\draw[dash pattern=on 4pt off 3pt, line width=0.4pt] (0,0) -- (\linewidth,0);}%
}

\newtheorem{theor}{Theorem}[section]

\theoremstyle{definition}
\newtheorem{claim}[theor]{Claim}

\newtheorem*{convention}{Convention}
\newtheorem{prop}[theor]{Proposition}%[section]
\newtheorem{proposition}[theor]{Proposition}%[section]
\newtheorem{lemma}[theor]{Lemma}%[section]
\newtheorem{cor}[theor]{Corollary}%[section]
\newtheorem{conjecture}[theor]{Conjecture}%[section]
\newtheorem{define}{Definition}[section]

\newtheorem{notation}{Notation}
\newtheorem{problem}{Problem}%[section]
\newtheorem{open}[problem]{Open problem}

\newtheorem{ex}{Example}[section]
\newtheorem*{implementNo}{Implementation}%[section]
\theoremstyle{definition}
\theoremstyle{definition}
\theoremstyle{remark}
\newtheorem{rem}{Remark}[section]
\newtheorem*{remNo}{Remark}

\newcommand{\BBR}{\mathbb{R}}
\newcommand{\BBN}{\mathbb{N}} 
\newcommand{\BBS}{\mathbb{S}}
\newcommand{\BBZ}{\mathbb{Z}}

\newcommand{\veps}{\varepsilon}

\DeclareMathOperator{\Vect}{Vect}

\DeclareMathOperator{\Wronsk}{Wronsk}

\DeclareMathOperator{\consta}{\boldsymbol{c}}

\newcommand{\constentry}[4]{%
  \noindent $p = #1$\par
  \noindent $N = #2$\par
  \noindent $\consta%\mathrm{const}
  (p) = $\seqsplit{#3}\par
  \noindent\textbf{Prime factorization:} #4\par
  \vspace{0.5em}
  \dashedhrule
  \vspace{0.5em}
}

\newcommand{\ta}{\mathcal{T}}

\hypersetup{
    colorlinks=true,
    linkcolor=blue,
    citecolor=blue,
    urlcolor=blue
}
\newcommand{\seqnum}[1]{\href{https://oeis.org/#1}{\underline{#1}}}

\newcommand{\by}[1]{\textit{{#1}}}
\newcommand{\jour}[1]{\textit{{#1}}}
\newcommand{\vol}[1]{\textbf{{#1}}}

\title[Fast calculation of OEIS\:A392\,714]{New integer sequence OEIS\:A392\,714 counts Wronskians:\\[3pt] fast evaluation via late\/-\/growing permutations}

\author[K.\,C.\,Shah]{Kian C.\ Shah}
\thanks{\textit{Address:} Bernoulli Institute for Mathematics,
  Computer Science and Artificial Intelligence, University of Groningen,
  P.O.\ Box 407, 9700\,AK Groningen, The Netherlands.}
\author[A.\,V.\,Kiselev]{Arthemy V.\ Kiselev${}^{\text{*}}$}
\thanks{${}^{\text{*}}$\:Corresponding author.\quad
  \textit{E-mail:} %  \href{mailto:a.v.kiselev@rug.nl}
  {\texttt{a.v.kiselev\symbol{"40}rug.nl}}}



\date{6 October 2026}
\begin{document}
\begin{abstract}%\noindent%
The alternating composition of $N = 2p$ weighted differential
operators $w_j(x)\cdot\partial_x^{\,p}$ of strict order $p$ on the line
$\mathbb{R} \ni x$ is again an operator of order $p$; its co\-ef\-fi\-ci\-ent
is the universal constant $\consta(p)$ times the Wronskian of the weights
$w_1,\ldots,w_N$. Lie brackets of vector fields fix $\consta(p=1)=1$; we want to find $\consta(p \geqslant 2)$:
e.g., $\consta(2) = 2$ or $\consta(3) = 90$. Direct
symbolic expansion (over $|S_{2p}| =(2p)!$ permutations) fails for $p \geqslant 4$.
Taking the monomials $w_j = x^{j-1}$ reduces the sum\-ma\-tion to
the much smaller set $\Phi_p \subseteq S_{2p-1} \subsetneq S_{2p}$ of late\/-\/growing per\-mu\-ta\-ti\-ons.
   %now $\consta(p)$ is expressed 
Expressing $\consta(p)$ as a signed sum of products of falling factorials,
we im\-ple\-ment and speed up the algorithm that gains all the integer values up to
$\consta(18) = 4.881\ldots \cdot 10^{462}$. 
   %To speed up the algorithm, we use a backtracking generator, memoized recursion,
   %and mod-primes accumulation of $\consta(p)$ via Chinese
   %Remainder Theorem. We thus gain all the values $\consta(p \leqslant 18)$, 
   %e.g., 
   %$\consta(15) = 5.046\ldots \cdot 10^{288}$. 
The resulting sequence is new, now registered as OEIS\:A392\,714;
its (sub)lea\-d\-ing\/-\/or\-der growth rate is
$\log \consta(p) \simeq 2p^2\log p -b p^2 + \overline{o}(p^2)$ for $p\gg 1$, 
with $b\geqslant 2.6744$.
\end{abstract}
\maketitle

\tableofcontents

\subsection*{Introduction}\label{sec:intro}
\noindent Vector fields $\vec{w} = w(x)\cdot\partial_x$ with smooth coefficients on
the real line $\mathbb{R} \ni x$ form the Lie algebra with respect to the
bracket
$[\vec{w}_1, \vec{w}_2]
 = \sum_{\sigma \in S_2} (-1)^{\sigma}\,
   \vec{w}_{\sigma(1)} \circ \vec{w}_{\sigma(2)}$,
where $\circ$ denotes the associative composition of differential
operators (here, of order $p = 1$). The arity of the Lie bracket is
$N = 2p = 2$, so that $S_2$ here is the permutation group of two elements
and $(-1)^{\sigma}$ denotes the sign of a permutation $\sigma \in S_N$.
The Lie bracket of vector fields satisfies the Jacobi identity, which is
immediately seen for these differential operators
$\vec{w}_j \in \mathrm{D}_1(\mathbb{R})$ of order $p = 1$.

Binary Lie algebra structures $[\cdot,\cdot]$ are deformed into strongly
homotopy (SH) Lie algebras; this yields linear combinations of totally
antisymmetric brackets of various arities $N \geqslant 2$. The formal
deformations (of the Lie structure $[{\cdot},{\cdot}]$ as the leading term% 
%in $\Delta$
),
\[
  [\cdot,\cdot] \;\longmapsto\;
  \Delta \;=\; [\cdot,\cdot]_{N=2} + [\cdot,\cdot,\cdot]_{N=3}
  + [\cdot,\ldots,\cdot]_{N=4} + \cdots,
\]
satisfy the countable set of Jacobi-type identities (see% Lada and Stasheff
~\cite{LadaStasheff,Stasheff2019Review} and references therein). The case of only
even arities $N = 2p$ for $p \in \mathbb{N}_{\geqslant 1}$ is of particular
interest both in Lie theory and physics. Indeed, for arbitrary even
arities, the alternated compositions of elements in an associative
algebra $\mathcal{A}$ do provide the multi-linear antisymmetric brackets that
satisfy the strongly homotopy Lie algebra's Jacobi identities 
(see~\cite{DzhumaAssoc,HanlonWachs}).
   %\cite{Kiselev2005AssociativeWronskians} and references therein).
We study the realisations of such SH-Lie structures by the associative
algebras $\mathrm{D}_p(\mathbb{R})$ of differential operators of strict order
$p \geqslant 1$ on the affine real line (\cite{Kiselev2005AssociativeWronskians,Kiselev_2025}); 
in the present paper%this work
, the brackets' arities are even positive integers,\footnote{\label{FootJacobiWron}
The Wronskians over $\mathbb{R}^1$
are known to satisfy the SH-Lie algebra's Jacobi identities for
arbitrary, i.e.\ even or odd, numbers of arguments~\cite{DzhumaWron}.
The multivariate, complete generalised Wronskians for functions in many
variables on $\mathbb{R}^{d > %\geqslant 
1}$ do again satisfy all the Jacobi
identities of SH-Lie algebras~\cite{Kiselev2005AssociativeWronskians};
this remains true for the incomplete generalised Wronskians over
$\mathbb{R}^{d > %\geqslant 
1}$ (see~\cite{Kiselev2025Multidim}).} 
%%%
$N = 2p \in \mathbb{N}$. It is readily seen
(Theorem~\ref{thm:main} on p.~\pageref{thm:main} below) that the arising
$N$-ary brackets act on the operators' coefficients $w_j(x)$ as Wronskian
determinants times the universal constants
$\consta(p) \in \mathbb{Z}$. Clearly, $\consta(1) = 1$.
Counting the Wronskians of size $2p \times 2p$ for $p \geqslant 1$, the
numbers $\consta(p)$ form an integer number sequence; constructed
in the most natural way, this sequence is, surprisingly, new. We discover
this by calculating its terms as far as $\consta(p = 18)$.

\smallskip%\S2.
\textbf{\textit{Problem outline.}}%Problem.
\quad The problem which we solve in this paper is: to design a fast algorithm for computation of the integers $\consta(p)$ for $1< p \lesssim 20$,
and to actually find the exact values as far as possible.
We reach $\consta(p=18)$ expressed by using 463 decimal digits.
We discover that the positive integer sequence $\bigl\{ \consta(p) \bigr\}_{p\in\BBN}$ is new; its observed growth rate is
$\log \consta(p) = \alpha p^2\log p + \beta p^2 + \overline{o}(p^2)$ with $\alpha\approx 2$ and $\beta \approx -2.67$ on $p \gtrsim 10$.
The known integers $\consta(p)$ have remarkable prime decompositions, the law and properties of which are yet to be explained: 
the nucleus of small primes (${}\leqslant p$) in high powers and a few primes in the intermediate range (between $p$ and $2p$) are mixed with very few big primes (${}\sim 10^{n\cdot 100}$ for $p \gtrsim 10$); e.g., the factorisation of $\consta(p=14)$ gave us a new confirmed prime of 206 decimal digits (see App.~\ref{app:data}).
%(these, we expect, might be used to generate very big primes${}\sim 10^{n\cdot 100}$
%from the tabulated composite numbers $\consta(p)$ at $p \gtrsim 10$).

\smallskip
\textbf{\textit{Method.}}%Method.
\quad The direct approach to finding the values $\consta(p)$ --\,by alternating $p$th order differential operators on the real line\,-- is unfeasible %fails
for $p\geqslant 4$. We reduce the problem from analysis to that in arithmetic and combinatorics by taking the monomials $1$,\ $x$,\ $\ldots$,\ $x^{2p-1}$ as operators' weights, whence we express the sought values $\consta(p)$ in terms of falling factorials (see Eq.~\eqref{eq:closedform} on p.~\pageref{eq:closedform} below).
The alternation over the permutation group $S_{2p}$ is by then reduced to the much smaller set $\Phi_p \subseteq S_{2p-1}\subsetneq S_{2p}$ 
of late\/-\/growing permutations.\footnote{\label{FootEvenOddPhiP}
We observe (and prove rigorously in~\cite{OddEven}) that the number of 
even and odd permutations in $\Phi_p$ differ by~$(-1)^{p+1}$.%
%we could not find this law in the available literature.
}

\smallskip
\textbf{\textit{Result.}}%Result.
\quad Our study has so far resulted in not only the values $\consta(p\leqslant 18)$ but also in the algorithm to compute them (implemented in \textsf{Python} and \textsf{C${+}{+}$}, see~\cite{ShahCode2026}; separately also in \textsf{Julia}, see~\cite{PetrellaJulia}) which underwent many iterations of optimisation for speed and memory allocation. The algorithm now relies on the backtracking (to avoid generating the entire set $\Phi_p$), on memoization (to avoid recalculation of the products of falling factorials), and on the Chinese Remainder Theorem (to avoid the arithmetic of integers${}\sim 10^{n\cdot 100}$ in a naive accumulation of the big numbers
$\consta(p\gtrsim 10)$, by now using the economical summation modulo several primes).
The observed growth law $\log\consta(p) = 2 p^2 \log p + O(p^2)$ tells us in advance the sufficient number of primes${}> 2p$ modulo which the fast arithmetic uniquely determines the value $\consta(p)$ at the end of the day.
The implementation of thus\/-\/accelerated algorithm can be run on powerful computing resources, allowing the discovery of subsequent values~$\consta(p\geqslant 19)$
and new big primes in their factorisation.

%Our contribution.

\smallskip
\textbf{\textit{Relation to other work%Significance of the result
.}}%Significance.
\quad The problem of counting the Wronskians --\,as they stem from the alternation of $p$th order differential operators\,-- is natural in differential calculus in the broadest sense; this counting problem could have been posed during the study of `infinitesimal groups' by Lie, when vector fields were recognised as differential operators of order $p=1$ with vanishing free term, and when their commutator was written to grasp the (non)\/permutability of two %$N=2p$ 
flows.
   %Differential-geometric interpretation of the alternation.
By counting the Wronskians at orders $p>1$ and arities $N=2p > 2$, we detect that 
the growth rate of the constants is $\consta(p) \sim p^{2p}$ for big~$p$.
Knowing this law is important in Theoretical Physics, specifically in the models which are based on the (strongly) homotopy deformations of Lie algebras (\cite{LadaStasheff,Stasheff2019Review}): the super\/-\/exponential rate
($\consta(1)=1$ and $\consta(2)=2$ but $\consta(3)=90$ and $\consta(4)=586\,656$) could dictate that collective effects are effectively suppressed for arities~$N>4$.

\begin{remNo}[{$\consta(p)$ at the crossroads}]\label{RemZaboronskiIHES}
The integers $\bigl\{ \consta(p) \bigr\}_{p\in\mathcal{I}\subsetneq\BBN}$ and formulas expressing the $p$th element $\consta(p)$ --\,other than equality~\eqref{eq:main} which defines the sequence and formula~\eqref{eq:closedform} which we use to obtain all the known values\,-- have been recognised in other domains of mathematics:
e.g., in combinatorics (of graphs and their multi\/-\/colouring problems) as well as in differential geometry (where the integers $\consta(p)$ count the top intersection number of Chern classes on a complete flag manifold $F\ell_{2p}$).
We owe these interpretations to O.\,Zaboronski 
(private communication).\footnote{\label{FootZabor}
\by{O.\,Zaboronski} (June 2026) $\consta(p)$ as the solution to a counting problem.
IH\'ES, 7~p. (work in progress), unpublished.}
%%%
Every such domain contributes with its toolkit and intuition, resulting so far in Claim~\ref{ClaimPositiveGrows} about basic properties of $\consta(p)$
(see p.~\pageref{ClaimPositiveGrows} below) 
and in Claim~\ref{ClaimAsympt} (on p.~\pageref{ClaimAsympt} below) about the asymptotic growth law for $\consta(p)$ as $p\to +\infty$. The latter claim itself relies on the findings which we report in this paper; on our part, we observe the predicted growth rate on the argument values $p\gtrsim 10$, cf.\ Table~\ref{TabBestFit} on p.~\pageref{TabBestFit}.
Furthering these lines of research will be the object of future work.\\
\centerline{\rule{1in}{0.7pt}}
\end{remNo}

%\smallskip
%\textbf{\textit{Perspectives.}}%Perspectives.
%\quad
  % N odd; base dimension d > 1.

\noindent%
The rest of this paper is structured as follows. In the next %immediately following 
section %, `Preliminaries', 
we show how the new integer sequence
$\consta(p)$ (\seqnum{A392714}) naturally arises from the
alternating of $N = 2p$ weighted differential operators of strict order $p$
on the real line.
In Table~\ref{tab:intro} on p.~\pageref{tab:intro} we report the previously known values $\consta(1) = 1$, $\consta(2) = 2$, and
$\consta(3) = 90$, and we recall some useful properties of the
integer sequence $\{\consta(p)\}_{p \in \mathbb{N}}$; these have been
established independently by Zaboronski${}^{\text{\ref{FootZabor}}}$ %(private communication) 
using techniques about large deviation models, standard in probability theory and
extremal combinatorics.
Then, in this paper we report as many values of $\consta(p)$ as
possible, currently attaining $p = 18$; we focus on accelerating every algorithm %of
for their computation.

Specifically, in Section~\ref{sec:monomials} we restrict the weights of the
differential operators on $\mathbb{R}^1$ to the monomials
$1, x, \ldots, x^{N-1}$, effectively transforming a problem of analysis
into arithmetic and combinatorics within the permutation group $S_{2p}$.
This is where the falling factorials appear. In Section~\ref{sec:phip} we
notice that after the monomial reduction, only a tiny fraction of
permutations contribute to the values $\consta(p)$; we recognize the
actually contributing set $\Phi_p \subset S_{2p-1} \subsetneq S_{2p}$ as the late\/-\/growing
permutations, and we provide the closed formula, Eq.~\eqref{eq:closedform}, 
for $\consta(p)$ in terms of the falling factorials and with alternated summation over
$\Phi_p$. We also discover a curious property of the set $\Phi_p$,
as seen from the last column of Table~\ref{tab:phip} on p.~\pageref{tab:phip};
we prove it for all $p\in\BBN$ in our subsequent paper~\cite{OddEven}.

The idea of the currently deployed algorithm for %the 
computation of $\consta(p)$ --\,more recent compared with the old algorithm, based
on calculations in the permutation group algebra (see
\cite{ShahKiselev2026})\,-- is outlined in Section~\ref{sec:algorithms} by
using pseudo-code; in Example~\ref{ex:peq2} we illustrate it for $p = 2$
(giving $\consta(2) = 2$ in terms of the falling factorials). From
this moment onwards, we focus on accelerating every element of the
algorithm for the actual computation\footnote{\label{FootWhereProgCodes}
The resulting code package, which implements Algorithm~\ref{alg:deployed},
is publicly available \cite{ShahCode2026}.
 %and is shipped as ancillary files with the arXiv version of this paper.
For convenience, in Appendix~\ref{app:oldcode} we keep the old, short code
which refers to the permutation group algebra approach, i.e.\ to the
calculation of linear combinations of permutations.}
%%%
of $\consta(p)$ as $p$ grows.

The resulting values $\consta(p)$ for $1 \leqslant p \leqslant 18$ are
reported in Table~\ref{tab:values} on p.~\pageref{tab:values} in
Section~\ref{sec:results}, continued in Appendix~\ref{app:data}. For each
value of $p$, we report the prime factorisation, up to one composite cofactor for $p=13$ and $15\leqslant p\leqslant 18$.\footnote{\label{FootMaybePrime}
The observed growth rate,
$\log \consta(p) \approx \alpha\, p^2 \log p + \beta\, p^2 + \cdots$
with $\alpha \approx 2$ and $\beta \approx -2.67$,
is so fast --\,and the structure of the prime decomposition of the integers
$\consta(p)$ is so peculiar (small primes in high powers and very
few big %huge 
primes in small powers)\,-- that it is currently impossible to
know the \emph{largest} prime factor for the values $\consta(p=13)$ and $\consta(p\geqslant 15)$; from the factorisation of $\consta(p=14)$ we extract a new
%, previously unknown
206-\/decimal\/-\/digit prime number.%
%assure that all the known factors are truly prime.
}
%%%
In our subsequent work~\cite{ShahKiselev2026} we explore these $p$-adic valuations of the integers $\consta(p)$, and we match the observed growth rate with the
predicted asymptotic behaviour of $\consta(p)$ as $p \to +\infty$ (by
Zaboronski${}^{\text{\ref{FootZabor}}}$%, work in progress
); in particular, Zaboronski confirms
the leading order value $\alpha = 2$ and gives a sharp bound for the
sub-leading coefficient, $\beta \leqslant -2.6744$.

We conclude this text with a set of open problems; we mention the immediate
next tasks and questions, and we outline possible directions of further
research.

\begin{convention}
Throughout this paper, we adopt the one-line notation for permutations, in
which $\sigma = (\sigma(1), \sigma(2), \ldots, \sigma(N))$ records the
image at %of 
each position. That is, $\sigma(k)$ denotes the index of the
object occupying position $k$ after the permutation is applied. Under this
convention, the identity element of $S_N$ is represented as
$\sigma = (1, 2, \ldots, N)$. For instance, in $S_4$, the identity
$\sigma = (1,2,3,4)$ corresponds to the map
$%\[
  (1,\, x,\, x^2,\, x^3) \;\longmapsto\; (1,\, x,\, x^2,\, x^3)
$,%\]
while $\sigma = (4,1,2,3)$ corresponds to
$%\[
  (1,\, x,\, x^2,\, x^3) \;\longmapsto\; (x^3,\, 1,\, x,\, x^2)$.
%\]
\end{convention}

\section{Preliminaries}\label{sec:prelim}
\noindent%
Let $p \in \mathbb{Z}_{\geqslant 1}$, and set $N = 2p$; let $w_1$,\ $\ldots$,\ $w_N$
be the coefficients of differential operators of strict order $p$ on an %the
interval $I \subseteq \mathbb{R}$ with an affine coordinate~$x$. For a smooth
function $f \in C^\infty(\mathbb{R})$, define the \emph{alternating
composition} of these differential operators:
\begin{equation}\label{eq:opdef}
  \mathcal{A}_p[w_1,\ldots,w_N](f)(x) \mathrel{{:}{=}}
  \sum_{\sigma \in S_N} (-1)^\sigma\,
  w_{\sigma(1)}\partial_x^p \circ w_{\sigma(2)}\partial_x^p
  \circ \cdots \circ w_{\sigma(N)}\partial_x^p\,(f(x));%,
\end{equation}%where 
the sum runs over all permutations $\sigma$ of the set
$\{1,\ldots,N\}$ and $(-1)^\sigma$ denotes the sign of~$\sigma$.

\begin{theor}[{\cite{Kiselev2005AssociativeWronskians,Kiselev_2025}}]\label{thm:main}
For every smooth function $f$ and all weight functions $w_1, \ldots, w_N$,
\begin{equation}\label{eq:main}
  \mathcal{A}_p[w_1,\ldots,w_N](f)(x)
  = \consta(p) \cdot
  \Wronsk \bigl(w_1, \ldots, w_N\bigr)(x) \cdot \partial_x^p(f(x)),
\end{equation}
where $\consta(p)$ is a universal constant, independent of $f$ and
of the weight functions.
\end{theor}

%Sketch of the proof.

The proof of this theorem is outlined in Appendix~\ref{app:proof}. The
problem which we study in this paper is the calculation of the constant's
values $\consta(p)$ in identity~\eqref{eq:main}.

\begin{lemma}\label{lem:order}
In every nonvanishing term of \eqref{eq:opdef}, the function $f$ appears
differentiated to precisely order $p$. Consequently, $\partial_x^p(f)$
factors out of the entire alternating sum, and \eqref{eq:main} reduces to
the following equation:
\begin{equation}\label{eq:reduced}
  \sum_{\sigma \in S_N} (-1)^\sigma\, w_{\sigma(1)}\partial_x^p \circ
  w_{\sigma(2)}\partial_x^p \circ \cdots \circ w_{\sigma(N)}
  \;=\; \consta(p) \cdot \Wronsk(w_1, \ldots, w_N).
\end{equation}
\end{lemma}

For small values $p \leqslant 3$ we verify identity \eqref{eq:main} by hand: we do so for $p = 1$, whence $\consta(1) = 1$, and for $p = 2$, whence we get
$\consta(2) = 2$ (cf.\ Example~\ref{ex:peq2} below). 
For $p = 3$, using \textsf{Maple} (e.g., in the package \textsf{Jets} for differential calculus% on diffieties
~\cite{BaranMarvan_Jets,Marvan2009}), 
we expand the left\/-\/hand side of \eqref{eq:main} by taking the entire sum over
$6! = 720$ permutations (here $N = 2 \cdot 3 = 6$ and all differential
operators have order $p = 3$); we find $\consta(3) = 90$. The values
of $\consta(p)$ for small $p$ are summarized in Table~\ref{tab:intro}.\footnote{\label{FootMapleCode123}
The plain \textsf{Maple} code to compute $\consta(p)$ up to $p=3$ is 
contained in %available from 
Appendix~\ref{app:maplecode} on p.~\pageref{app:maplecode}.}
%%%
\begin{table}[htb]
\caption{Previously known values of $\consta(p)$ for $p = 1, 2, 3$.
The case $p = 1$ recovers the Lie bracket of two vector fields. Extended
data appear in Table%s~\ref{tab:phip} and
~\ref{tab:values} on p.~\pageref{tab:values} and in
Appendix~\ref{app:data}.}\label{tab:intro}
\centering
\renewcommand{\arraystretch}{1.4}
\begin{tabular}{ccl}
\hline\hline
$p$ & $N!$ & $\consta(p)$ \\
\hline
$1$ & $2$ & $1$ \quad (Lie) \\
$2$ & $24$ & $2$ \\
$3$ & $720$ & $90$ \\
\hline\hline
\end{tabular}
\end{table}
%%%%%%%%%%%%%%%%%%%
  %%% MAPLE program code for JETS ?

\begin{claim}[O.\,Zaboronski%, private communication
${}^{\text{\protect\ref{FootZabor}}}$]\label{ClaimPositiveGrows}
(\textit{i}) %\quad 
The universal constants $\consta(p)$ in equality~\eqref{eq:main} are positive: $\consta(p)>0$ for all $p\in\BBN_{\geqslant 1}$.\\[0.5pt]
%%%
(\textit{ii}) %\quad  
The integer sequence $\bigl\{ \consta(p) \bigr\}_{p\in\BBN}$ %is monotonically growing
grows monotonically: $1 \leqslant \consta(p) < \consta(q)$ if $1\leqslant p < q$.
\end{claim}
   %More about the asymptotic behaviour: later on p.~...;

However simple be the definition of $\consta(p)$, its calculation at
$p = 4$ becomes infeasible using \textsf{Maple}
(in particular, \textsf{Jets}) if one still tries a straightforward
expansion of both sides in \eqref{eq:main} or \eqref{eq:reduced}, then
solving for the proportionality constant. To learn
$\consta(p \geqslant 4)$, we need a new, more economical strategy.

\section{Combinatorics of the monomial case}\label{sec:monomials}
\noindent%
To make the right\/-\/hand side of~\eqref{eq:reduced} shorter and extract the
constants $\consta(p)$ explicitly, we choose specific weight
functions, namely, let them be monomials:
\begin{equation}\label{EqMonomWeight}
  w_n(x) = x^{n-1}, \qquad n = 1, \ldots, N.
\end{equation}

\begin{rem}
Note that $\deg(w_{\sigma(k)}) = \sigma(k) - 1$ whenever $w_k = x^{k-1}$.
\end{rem}

\begin{lemma}\label{lem:wronskian}
The Wronskian of $1, x, x^2, \ldots, x^{N-1}$ equals
\[
  \Wronsk(1,\, x,\, x^2,\, \ldots,\, x^{N-1})
  \;=\; \prod_{k=0}^{N-1} k!\,.
\]
\end{lemma}

\begin{proof}%We note that 
The Wronskian matrix of $1, x, \ldots, x^{N-1}$ is upper
triangular, with diagonal entries $0!$,\ $1!$,\ $2!$,\ $\ldots$,\ $(N-1)!$. Thus, the
determinant is the product of these entries.
\end{proof}

To understand the action of the iterated operators in \eqref{eq:opdef} in
the monomial case, consider two monomials $x^\alpha$ and $x^\beta$ with
$\alpha, \beta \in \mathbb{Z}_{\geqslant 0}$. The operator $x^\beta\partial_x^p$
acts on $x^\alpha$: whenever $\alpha < p$ the result is zero, otherwise
\begin{equation}\label{eq:opaction}
  x^\beta\,\partial_x^p\,(x^\alpha)
  \;=\; x^\beta\cdot\frac{\alpha!}{(\alpha-p)!}\, x^{\alpha-p}
  \;=\; \frac{\alpha!}{(\alpha-p)!}\, x^{\alpha+\beta-p}.
\end{equation}
Each application of the weighted operator $x^\beta\partial_x^p$, when
nonzero in the output, multiplies the argument $x^\alpha$ by the \textsl{falling
factorial}, for which the notation is
\[
  \alpha^{\underline{p}}
  = \frac{\alpha!}{(\alpha-p)!}
  = \alpha(\alpha-1)\cdots(\alpha-p+1);
\]
also, the operator $x^\beta\partial_x^p$ shifts the exponent by
$\alpha \mapsto \alpha + \beta - p$.

\begin{rem}
The vanishing condition $\alpha < p$ immediately implies that not every
permutation $\sigma \in S_N$ contributes to the alternating sum in~\eqref{eq:reduced}. Any permutation for which the innermost weight function
$w_{\sigma(N)} = x^{\sigma(N)-1}$ satisfies $\sigma(N) - 1 < p$ yields zero
contribution; clearly, such a permutation may be discarded at the outset.
\end{rem}

\begin{ex}\label{ex:filter}
Let $p = 2$ and $N = 4$, so that the weight functions are
$\{1, x, x^2, x^3\}$. Consider the permutation $\sigma = (3,1,4,2)$,
corresponding to
$(1,\, x,\, x^2,\, x^3) \mapsto (x^2,\, 1,\, x^3,\, x)$,
so that $\sigma(4) = 2$ and the innermost operator acts as
$\partial_x^2(w_2) = \partial_x^2(x) = 0$. The entire term arising from
this permutation thus vanishes, and this $\sigma$ does not contribute to the
sum. More generally, any permutation with $\sigma(N) \leqslant p$ is eliminated
by this filter; for $p = 2$, $N = 4$ this removes all $\sigma$ with
$\sigma(4) \in \{1, 2\}$, reducing the number of permutations to be
evaluated from $4! = 24$ to those with $\sigma(4) \in \{3, 4\}$.
\end{ex}

As illustrated by the above example, we can now shift our focus %ENG.
from the monomial weight functions to the permutations themselves; from now on, we
work with permutations directly.

\begin{lemma}\label{lem:sigma1}
For every contributing permutation $\sigma \in S_N$, we have
$\sigma(1) = 1$\textup{;} that is, the weight function $w_{\sigma(1)} = w_1 = 1$ is
necessarily fixed in the leading leftmost position.
\end{lemma}

\begin{proof}
There are two kinds of permutations: the ones which keep the leftmost
monomial $w_1 = 1$ in its first place, and those which send it away, so
that a positive-degree monomial $x^{\alpha}$ becomes the leftmost. Suppose
the latter, so that it is this positive-degree monomial $x^{\alpha}$ which
itself is not differentiated, but is multiplied by a nonnegative-degree
monomial resulting from all the differentiations which stood to the right
of it. This yields a contribution of degree $\geqslant \alpha$ to the sum. But
we know that in the right-hand side of \eqref{eq:reduced}, the Wronskian
of $1, x, \ldots, x^{N-1}$ is a real number, i.e.\ it is of polynomial degree
zero. Therefore, every permutation sending $1$ away cannot contribute to
the left-hand side of \eqref{eq:reduced}.
\end{proof}

  %closed form for
\section{Late\/-\/growing permutations and a formula of
\texorpdfstring{$\consta(p)$}{const(p)}}
\label{sec:phip}

\noindent To identify all the permutations (from the full set $S_N$ of size $N!$)
that do contribute to the sum in the left-hand side of \eqref{eq:reduced},
we combine Lemma~\ref{lem:sigma1} with the exponent-tracking process of
formula \eqref{eq:opaction}, read from right to left.

\begin{define}\label{def:phip}
For $\sigma \in S_N$ define the partial sums
\begin{equation}\label{eq:Tk}
  T_k \;:=\; \sum_{i=N-k+1}^{N} \bigl(\sigma(i) - 1 - p\bigr),
  \qquad k = 1, \ldots, N-1.
\end{equation}
The quantity $T_k$ tracks whether the exponent of $x$ remains nonnegative
after each successive operator application: if $T_k < 0$ for some $k$, the
exponent becomes negative at that stage, and the term vanishes.\\[0.5pt]
$\bullet$\quad 
We denote by $\Phi_p$ the set of all permutations $\sigma \in S_N$ with
$\sigma(1) = 1$ and $T_k \geqslant 0$ for all $k = 1, \ldots, N-1$; these are
exactly the permutations that contribute to the left-hand side of
\eqref{eq:reduced}, and we claim that $\Phi_p \subsetneq S_N$.
\end{define}

For the special choice of monomial weights
$w_1 = 1, \ldots, w_N = x^{N-1}$ we therefore have
\[
  \sum_{\sigma \in S_N} (-1)^\sigma\, w_{\sigma(1)}\partial_x^p \circ
  w_{\sigma(2)}\partial_x^p \circ \cdots \circ w_{\sigma(N)}
  \;=\;
  \sum_{\sigma \in \Phi_p} (-1)^\sigma\,
  w_{\sigma(1)}\partial_x^p \circ w_{\sigma(2)}\partial_x^p
  \circ \cdots \circ w_{\sigma(N)};
\]
let us emphasize that the set $\Phi_p$ was constructed in reference to this
choice of weights~$\{w_j\}$.

\begin{ex}\label{ex:p1p2}
We illustrate the filtering condition for $p = 1$ and $p = 2$:
in fact, $\Phi_p \subsetneq S_{2p}$.

\textbf{Case $p = 1$, $N = 2$.} The group $S_2$ consists of two
permutations, and thus the alternating sum \eqref{eq:reduced} reads
\[
  \sum_{\sigma \in S_2} (-1)^\sigma\, w_{\sigma(1)}\partial_x \circ
  w_{\sigma(2)}.
\]
The two terms are as follows. For $\sigma = (1,2)$, we obtain the term
$1 \cdot \partial_x(x)$, which does contribute. For $\sigma = (2,1)$, we
obtain the term $x \cdot \partial_x(1)$, which does not contribute because
$\partial_x(1) = 0$. This is also consistent with Lemma~\ref{lem:sigma1}:
since $\sigma(1) = 2 \neq 1$, this permutation is discarded. Thus
\[
  \Phi_1 = \{(1,2)\} \subset S_2, \qquad
  |\Phi_1| = 1 = \tfrac{1}{2}\,|S_2|.
\]

\textbf{Case $p = 2$, $N = 4$.} The group $S_4$ has $4! = 24$ elements.
Applying Lemma~\ref{lem:sigma1} immediately reduces it to the $3! = 6$
permutations with $\sigma(1) = 1$. Of these, the nonnegativity condition on
the sums $T_k$ eliminates three more, leaving the valid set
\[
  \Phi_2 = \{(1,2,3,4),\ (1,2,4,3),\ (1,3,2,4)\} \subset S_4,
  \qquad |\Phi_2| = 3 = \tfrac{1}{8}\,|S_4|.
\]
The three discarded permutations with $\sigma(1) = 1$ are $(1,3,4,2)$,
$(1,4,2,3)$, and $(1,4,3,2)$; each of them produces a negative exponent at
some intermediate stage and therefore contributes a zero term to the sum.
\end{ex}

\begin{rem}\label{rem:idrev}
The permutations that are ``close to'' the identity permutation
$(1, 2, \ldots, N)$ contribute to the sum, while permutations that are
``close to'' the reversed order permutation $(N, N-1, \ldots, 2, 1)$ do
not contribute to the sum. Indeed, permutations near the identity present
the weight functions in basically increasing order of their degree, ensuring that
each intermediate exponent remains sufficiently large to survive $p$
differentiations. On the other hand, permutations near the reversed order
present lower\/-\/degree weights in the innermost operators, which has a higher
chance of driving the exponent below zero.
On the opening pages of our sequel~\cite{OddEven} to this text we make precise these intuitive ideas.
\end{rem}

\begin{rem}
It is clear that at $p=1$, the set $\Phi_1 \simeq S_1 = \{\text{id}\}$ amounts to just one permutation. Let us count their number in the larger sets~$\Phi_{p>1}$ at hand.
%%%
The sequence $|\Phi_p| = 1$,\ $3$,\ $35$,\ $1001$,\ $53109$,\ $4605271$,\ $\ldots$ for
$p = 1, 2, 3, 4, 5, 6$ coincides with the terms of odd index $N' = 2p-1$ of
the sequence \seqnum{A147681} in the {\it On-Line Encyclopedia of Integer
Sequences} \cite{oeisA147681}, which counts \emph{late-growing}
permutations\footnote{\label{FootDefLateGrow}
By definition, a permutation $\sigma \in S_n$ is \emph{late\/-\/growing} if its
partial sums never exceed the corresponding partial sums of the average
permutation, that is,
\[
  \frac{1}{k}\sum\nolimits_{i=1}^{k} \sigma(i) \;\leqslant\; \frac{n+1}{2}
  \qquad \text{for all $k \leqslant n$.}
\]
The bound on the right\/-\/hand side is the arithmetic mean of the numbers
$1$,\ $\ldots$,\ $n$; thus the ``average permutation'' has partial sums
$k(n+1)/2$ as $k\geqslant 1$ grows, so dividing by $k$ gives the stated inequality.} 
%%%
of $\{1, \ldots, n\}$. 
%%%
Some statistics about the structure of the sets $\Phi_p$ is summarized in Table~\ref{tab:phip}. %on p.~\pageref{tab:phip}.
We discover that as $p$ grows, our strategy to restrict the summation to
the subset $\Phi_p \subseteq S_{N-1} \subsetneq S_N$ becomes more and more effective. %the denominators in the approximations $|\Phi_p|/N!$ grow monotonically.
%%%
\begin{table}[htb]
\caption{For each $p$, the size of the contributing set
$\Phi_p \subset S_{N-1} \subsetneq S_{N}$ with $N = 2p$, its fraction of the size of the group $S_N$, its ratios to $(N-1)!$ and to $(N-2)!$, and the difference
between the numbers of parity\/-\/even and odd permutations in~$\Phi_p$. }\label{tab:phip}
\centering
\small
\renewcommand{\arraystretch}{1.25}
\begin{tabular}{crrrrc}
\hline\hline
$p$ & $|\Phi_p|$ & $|\Phi_p|/N!$ & $|\Phi_p|/(N-1)!$ & $|\Phi_p|/(N-2)!$ & $e_p - o_p$ \\[0.5pt]
\hline
$1$ & $1$ & $1/2$ & $1$ & $1$ & $+1$ \\[0.5pt]
\hline
$2$ & $3$ & $1/8$ & $1/2$ & $3/2$ & $-1$ \\
$3$ & $35$ & $\approx 1/20$ & $\approx 0.2917$ & $\approx 1.4583$ & $+1$ \\
$4$ & $1001$ & $\approx 1/40$ & $\approx 0.1986$ & $\approx 1.3903$ & $-1$ \\
$5$ & $53109$ & $\approx 1/68$ & $\approx 0.1464$ & $\approx 1.3172$ & $+1$ \\
$6$ & $4605271$ & $\approx 1/104$ & $\approx 0.1154$ & $\approx 1.2691$ & $-1$ \\
$7$ & $589809987$ & $\approx 1/147$ & $\approx 0.0947$ & $\approx 1.2313$ & $+1$ \\
$8$ & $104899483845$ & $\approx 1/199$ & $\approx 0.0802$ & $\approx 1.2033$ & $-1$ \\
$9$ & $24698838710457$ & $\approx 1/259$ & $\approx 0.0694$ & $\approx 1.1805$ & $+1$ \\
$10$ & $7444522779300351$ & $\approx 1/326$ & $\approx 0.0612$ & $\approx 1.1628$ & $-1$ \\
$11$ & $2792467448952670671$ & $\approx 1/402$ & $\approx 0.0547$ & $\approx 1.1478$ & $+1$ \\
$12$ & \seqsplit{1276369340371154144337} & $\approx 1/486$ & $\approx 0.0494$ & $\approx 1.1356$ & $-1$ \\
$13$ & \seqsplit{698008560075731437652425} & $\approx 1/577$ & $\approx 0.0450$ & $\approx 1.1250$ & $+1$ \\
$14$ & \seqsplit{450111715263775132783135875} & $\approx 1/677$ & $\approx 0.0413$ & $\approx 1.1161$ & $-1$ \\
\hline\hline
\end{tabular}
\end{table}
\end{rem}

\begin{conjecture}
As $p\to +\infty$, the ratio $|\Phi_p| \big/ (2p-2)!$ converges from above to~$1+0$.%
%%%
\footnote{\label{FootTwoSeqNumLateGrow}
Counting the number of late\/-\/growing permutations --\,of $m\in\BBN_{\geqslant 1}$ objects $\{1$,\ $\ldots$,\ $m\}$,\,-- elements of the integer sequence \seqnum{A147681} behave differently for odd indices $m=2p-1$ (as in this paper, meaning $|\Phi_p|$) and for even indices $m=2\ell$. Specifically, it is readily seen from~\cite{oeisA147681} that the ratios $|\Phi_p|/(2p-2)!$ tend from above to, hypothetically, $1+0$ \emph{much slower} than the analogous ratios of the even\/-\/indexed, $2\ell$th terms over $(2\ell-1)!$. Although the definition of late growth is uniform for any parity of the number of objects~$m$, the subsets and counts of late\/-\/growing permutations in the groups $S_{m=2p-1}$ and $S_{m=2\ell}$ form two subsequences, clearly unlike each other.}
%%%
%In relation to the number of permutations $|Phi_p|\subset S_{2p-1}$ of odd number of objects, we pick the \emph{odd}\/-\/indexed terms of the integer sequence \seqnum{A147681} (%in the {\it On-Line Encyclopedia of Integer Sequences} 
%\cite{oeisA147681}). 
%About the even\/-\/indexed terms we make an identical
  %, but separate 
%conjecture: as $p\to+\infty$, the limit 
%equals $1+0$
%, yet
   %Additionally, the even-indexed terms seem to fit well into this pattern as well. In fact, 
%we already see that the ratio of the numbers of even\/-\/indexed late growing permutations (of)
%%%%%%%%%%
%to (N-2)! starts converging earlier to 1 (attaching image of graphs) than the ratio on our paper.}
%%%
\end{conjecture}

\begin{rem}\label{obs:parity}
We note one striking structural feature of the sets~$\Phi_p$, namely
   %this data, concerning 
the distribution of even and odd permutations. % within $\Phi_p$. Namely, 
The counts of even and odd permutations in $\Phi_p$ always differ by exactly
$1$, with the dominant parity alternating with $p$: odd permutations
dominate for even $p$ and \textit{vice versa}. 
   %We record this as a conjecture for $p \geqslant 14$ onwards.
We claim this is always true (for $p \geqslant 15$ as well); a proof is contained in our subsequent text~\cite{OddEven}.
\end{rem}

For $\sigma\in\Phi_p$, we drop the forced leading $1$ and subtract $1$ from every remaining entry; the result is a late\/-\/growing permutation of \(\{1,\ldots,2p-1\}\). 

We now derive a closed\/-\/form expression for each term in the sum over
$\Phi_p$, yielding an explicit formula for $\consta(p)$.

\begin{notation}\label{def:Ek}
For an integer $k$ in the range $1 \leqslant k \leqslant N-1$ and a permutation
$\sigma \in \Phi_p$ we put, by definition,
\begin{equation}\label{eq:Ek}
  E_k \;:=\; \sum_{j=0}^{k-1}\bigl(\sigma(N-j) - 1\bigr) - (k-1)p,
\end{equation}
so that its falling factorial is
\[
  E_k^{\,\underline{p}} = \frac{E_k!}{(E_k - p)!}.
\]
\end{notation}

\begin{prop}\label{prop:closedform}
The constants $\consta(p)$ from Eq.~\eqref{eq:main} in Theorem~\ref{thm:main} satisfy the following equality:
\begin{equation}\label{eq:closedform}
  \consta(p) \cdot \prod_{k=0}^{N-1} k!
  \;=\;
  \sum_{\sigma \in \Phi_p} (-1)^\sigma \prod_{k=1}^{N-1}
  E_k^{\,\underline{p}}.
\end{equation}
\end{prop}

\begin{proof}
For our chosen set of monomial weights, each surviving term in the
left-hand side of \eqref{eq:reduced} is
\begin{multline*}
  \left(1\cdot\partial_x^p\circ w_{\sigma(2)}\cdot\partial_x^p\circ\cdots
  \circ w_{\sigma(N-1)}\cdot\partial_x^p\right)(w_{\sigma(N)})\\
  =\left(1\cdot\partial_x^p\circ x^{\sigma(2)-1}\cdot\partial_x^p\circ
  \cdots\circ x^{\sigma(N-1)-1}\cdot\partial_x^p\right)(x^{\sigma(N)-1}).
\end{multline*}
The right-hand side of formula \eqref{eq:closedform} tracks the scalar
coefficient accumulated as the composition of weighted operators acts on
the monomial $x^{\sigma(N)-1}$, proceeding from right to left through the
$N-1$ applications of $\partial_x^p$. We study this right-to-left action
in more detail.

\smallskip
\noindent\textbf{Rightmost action.} The rightmost operator acts on the
weight $w_{\sigma(N)} = x^{\sigma(N)-1}$:
\[
  \partial_x^p\!\left(x^{\sigma(N)-1}\right)
  = \frac{(\sigma(N)-1)!}{(\sigma(N)-1-p)!}\, x^{\sigma(N)-1-p}
  = E_1^{\,\underline{p}}\cdot x^{\sigma(N)-1-p},
\]
where $E_1 = \sigma(N)-1$. The exponent of $x$ is now $\sigma(N)-1-p$.

\smallskip
\noindent\textbf{Second action.} Pre-multiplying by
$w_{\sigma(N-1)} = x^{\sigma(N-1)-1}$ shifts the exponent to
\[
  (\sigma(N-1)-1) + (\sigma(N)-1-p)
  = \sigma(N-1)+\sigma(N)-2-p = E_2,
\]
and a further application of $\partial_x^p$ yields the factor
\[
  \frac{E_2!}{(E_2-p)!} = E_2^{\,\underline{p}}.
\]

\smallskip
\noindent\textbf{General step.} After $k$ actions of $\partial_x^p$, the
accumulated exponent of $x$ is $E_k - p$. Pre-multiplying by
$w_{\sigma(N-k)} = x^{\sigma(N-k)-1}$ raises the exponent to $E_{k+1}$, and
the $(k+1)$th application of $\partial_x^p$ contributes the factor
$E_{k+1}^{\,\underline{p}}$. Iterating through all the $N-1$ actions, the
total scalar coefficient accumulated for the permutation
$\sigma \in \Phi_p$ is
\[
  c(\sigma,p) \;=\; \prod_{k=1}^{N-1} E_k^{\,\underline{p}}.
\]
Summing over all the contributing permutations with their signs, and
equating the result to the product of $\consta(p)$ and the
Wronskian determinant
$\Wronsk(1,x,\ldots,x^{N-1}) = \prod_{k=0}^{N-1}k!$ from
Lemma~\ref{lem:wronskian}, yields equality \eqref{eq:closedform}.
\end{proof}

\section{Two algorithms for the evaluation of
\texorpdfstring{$\consta(p)$}{const(p)}}
\label{sec:algorithms}

\subsection{The algorithm: from filtering to memoized recursion}
\label{sec:oldalg}
To directly implement our algorithm, we would start by generating the set of all $N!$ permutations of $\{1, 2, \ldots, N\}$, discard any permutation with $\sigma(1) \neq 1$, and test the remaining permutations by simulating the exponent-tracking
process from right to left: the partial sums $T_k$ of \eqref{eq:Tk} are
accumulated, and if $T_k < 0$ for some $k$, the permutation is rejected.
Only those permutations for which $T_k \geqslant 0$ for all $k$ are retained;
by Definition~\ref{def:phip}, the retained set is exactly~$\Phi_p$.

However, this naive approach can instantly be made better by using a \emph{backtracking algorithm}, which avoids generating invalid permutations in the
first place. Rather than constructing all $N!$ permutations \textit{ab
initio} and filtering \textit{post hoc}, we build permutations incrementally via backtracking, enforcing the nonnegativity condition $T_k \geqslant 0$ at each step of the construction.

Concretely, the element $\sigma(1) = 1$ is fixed from the beginning. The
remaining positions are filled from right to left: at each stage, a
candidate element $R$ is appended to the current partial permutation only
if the updated running sum, incremented by $R - (p+1)$, remains
nonnegative. If it does not, that branch of the search tree is removed
entirely, and no permutations extending that partial assignment are ever
constructed. When the remaining pool of numbers $\{2, 3, \ldots, N\}$ is
exhausted, the completed suffix is prepended with $1$ to yield a valid
permutation.

The running sum maintained by the filter is precisely $E_k - p$ of
\eqref{eq:Ek}, so the same quantity that decides whether a branch survives
also supplies the factor $E_k^{\,\underline{p}}$ contributed at that step,
and the product $c(\sigma,p)$ and the sign $(-1)^{\sigma}$ accumulate
alongside $\sigma$ itself rather than being computed once it is finished.

Storing the values (which the recursion revisits) then turns this into the \emph{`memoized' recursion} of Algorithm~\ref{alg:old} below, which does not enumerate \(\Phi_p\) at all. We revisit the same values at multiple times in the recursion tree, and keeping these values memoized saves us the time of recomputing them at the expense of memory.

Algorithm~\ref{alg:old} summarizes the old algorithm\footnote{\label{rem:code}
  %\begin{rem}\label{rem:code}
The short reference implementation of Eq.~\eqref{eq:closedform} is listed in
Appendix~\ref{app:oldcode}. Its cache key is the triple
$(\mathit{avail}, s, d)$ of Algorithm~\ref{alg:old}: the ordered list
$\mathit{avail}$ of labels not yet placed, the running exponent
$s = E - p$ of the preceding step, and the number $d$ of labels already
placed.%
  %\end{rem}
}
%%%%
in pseudo-code: the backtracking generation of $\Phi_p$ and, along each completed branch, the accumulation of the signed product of falling factorials from
Proposition~\ref{prop:closedform}.

\begin{algorithm}[htb]
\caption{The old algorithm: backtracking construction of $\Phi_p$, with the
values of the recursion stored under the key
$(\mathit{avail}, s, d)$.}
\label{alg:old}
\begin{algorithmic}[1]
\Require $p \geqslant 1$; put $N := 2p$ and $W := \prod_{k=0}^{N-1} k!$
\State $\ta \gets$ empty table
  \Comment{the store; its keys are the triples $(\mathit{avail}, s, d)$}
\Function{B}{$\mathit{avail},\ s,\ d$}
  \Comment{$\mathit{avail} = (e_1 < e_2 < \cdots < e_m)$: the labels not yet
  placed, as an ordered list; $s$: the running exponent, equal to $E - p$ of
  the preceding step; $d$: how many labels have been placed}
  \If{$(\mathit{avail}, s, d) \in \ta$}
    \State \Return $\ta[(\mathit{avail}, s, d)]$
      \Comment{this value was already produced on another branch}
  \EndIf
  \If{$d = N-1$}
    \State \Return $1$
      \Comment{all of $\{2, \ldots, N\}$ has been placed}
  \EndIf
  \State $t \gets 0$
  \For{$i = 1, \ldots, m$}
    \State $E \gets s + (e_i - 1)$
      \Comment{the exponent value $E_{d+1}$ of \eqref{eq:Ek}}
    \If{$E \geqslant p$}
      \Comment{otherwise the branch is pruned; cf.\ Definition~\ref{def:phip}}
      \State $t \gets t + (-1)^{\,m-i}\, E^{\underline{p}}\,
        \Call{B}{(e_1, \ldots, \widehat{e_i}, \ldots, e_m),\ E-p,\ d+1}$
        \Comment{$\mathit{avail}$ is sorted, so $(-1)^{m-i}
        = (-1)^{\#\{f \in \mathit{avail} \,:\, f > e_i\}}$}
    \EndIf
  \EndFor
  \State $\ta[(\mathit{avail}, s, d)] \gets t$
  \State \Return $t$
\EndFunction
\State \Return $\bigl|\Call{B}{(2, 3, \ldots, N),\ 0,\ 0}\bigr| \,/\, W$
  \Comment{$= \consta(p)$; the division is exact by
  Prop.~\ref{prop:closedform}}
\end{algorithmic}
\end{algorithm}

\begin{ex}\label{ex:peq2}
We now illustrate Proposition~\ref{prop:closedform} for $p = 2$, $N = 4$.
The contributing permutations, as identified in Example~\ref{ex:p1p2}, are
\[
  \Phi_2 = \{(1,2,3,4),\, (1,2,4,3),\, (1,3,2,4)\}.
\]
We compute the coefficient
$c(\sigma, 2) = \prod_{k=1}^{3} E_k^{\,\underline{2}}$ for each filtered
permutation $\sigma \in \Phi_2$.

\smallskip
For $\sigma = (1,2,3,4)$, the sign is $(-1)^\sigma = +1$ and
\[
  E_1 = 3,\quad E_2 = 3+2-2 = 3,\quad E_3 = 3+1-2 = 2,
\]
\[
  c((1,2,3,4),\,2) = 3^{\underline{2}}\cdot 3^{\underline{2}}
  \cdot 2^{\underline{2}} = 3!\cdot 3!\cdot 2!.
\]

For $\sigma = (1,2,4,3)$, the sign is $(-1)^\sigma = -1$ and
\[
  E_1 = 2,\quad E_2 = 2+3-2 = 3,\quad E_3 = 3+1-2 = 2,
\]
\[
  c((1,2,4,3),\,2) = 2^{\underline{2}}\cdot 3^{\underline{2}}
  \cdot 2^{\underline{2}} = 3!\cdot 2!\cdot 2!.
\]

For $\sigma = (1,3,2,4)$, the sign is $(-1)^\sigma = -1$ and
\[
  E_1 = 3,\quad E_2 = 3+1-2 = 2,\quad E_3 = 2+2-2 = 2,
\]
\[
  c((1,3,2,4),\,2) = 3^{\underline{2}}\cdot 2^{\underline{2}}
  \cdot 2^{\underline{2}} = 3!\cdot 2!\cdot 2!.
\]

\smallskip
Assembling the signed sum and dividing by the Wronskian determinant of
Lemma~\ref{lem:wronskian}, namely
$\Wronsk(1,x,x^2,x^3) = 3!\cdot 2!\cdot 1!$, we obtain
\begin{align*}
  \consta(2)
  &= \frac{\overbrace{3!\cdot 3!\cdot 2!}^{c(\sigma_1,\,2)}
  - \overbrace{3!\cdot 2!\cdot 2!}^{c(\sigma_2,\,2)}
  - \overbrace{3!\cdot 2!\cdot 2!}^{c(\sigma_3,\,2)}}
  {\underbrace{3!\cdot 2!\cdot 1!}_{\Wronsk}}\\[6pt]
  &= 3! - 2! - 2! \;=\; 6 - 2 - 2 \;=\; 2,
\end{align*}
in agreement with the value $\consta(2) = 2$ recorded in
Table~\ref{tab:intro}.
\end{ex}

\subsection{Accelerate %ing 
the algorithm: memoization over subsets and
mod-primes arithmetic}%
  %arithmetic modulo primes}
\label{sec:memo}
%%%
The plain backtracking generator of Section~\ref{sec:oldalg} produces
the set of late\/-\/growing permutations $\Phi_p$ element by element, 
yet $|\Phi_p|$ grows super-exponentially with $p$ (see Table~\ref{tab:phip}): 
already at $p = 7$ there are $589\,809\,987$
contributing permutations, each carrying its own product of $N-1$ falling
factorials. It is precisely this cost that is reduced by memoization. Because the stored values are
keyed on the sorted list $\mathit{avail}$, two branches that have placed the
same labels in different orders meet at one and the same entry, and the
recursion of Algorithm~\ref{alg:old} visits $4\,332$ distinct states at
$p = 7$ whereas the generator visited $589\,809\,987$ permutations. What
Algorithm~\ref{alg:old} does not do is make those states cheap. The rest of
this section is about the cost of a single stored state: first its address,
then its content.

The observation that drives our next acceleration is contained in formula \eqref{eq:Ek}: the
exponent value $E_k$ depends on a permutation $\sigma$ only through the \emph{sum}
$\sum_{j=0}^{k-1}(\sigma(N-j) - 1)$ of the values placed so far, and this
sum is determined by the \emph{set} of placed values, not by the order in
which they were placed. Consequently, two partial permutations that have
placed the same set of values will have identical futures: the same candidates
remain, the same exponent values $E_k$ will be produced, and the same
factors $E_k^{\,\underline{p}}$ will be accumulated. Thus, the signed sum of Proposition~\ref{prop:closedform} satisfies a
recursion over subsets, which we will show in this subsection.

The generative procedure of Section~\ref{sec:oldalg} builds permutations $\sigma$ from the right, choosing the element $\sigma(N)$, then $\sigma(N-1)$, and so on, and rejecting a branch as soon as the running
exponent index falls below $p$. We again build permutations from right to left; by Lemma~\ref{lem:sigma1}, the
value $1$ is reserved for position~$1$. At any point in that construction, the labels $\{2,\dots,N\}$ are partitioned into those already placed and those still to be placed. We write
\[
  A \;\subseteq\; \{2,3,\dots,N\}
\]
for the set of labels still to be placed. To prove that our $E_k$ depends only on the set \(A\), we first prove a lemma for the shifted labels; we subtract \(p+1\) from the remaining labels to obtain new labels symmetric around 0 (e.g.\ \(\{2,3,4\}\rightarrow\{-1,0,1\}\) for \(p=2\)).
\begin{lemma}[The shifted labels sum to zero]\label{lem:zerosum}
\[
  \sum_{j=2}^{N}\bigl(j-p-1\bigr) \;=\;\sum_{j=2}^{N}j - (N-1)(p+1)\;=\; 0 .
\]
\end{lemma}
 
\begin{proof}
$\sum_{j=2}^{N} j = \tfrac{N(N+1)}{2}-1 = p(2p+1)-1 = 2p^2+p-1$ and the remaining part of the sum, \((N-1)\) copies of \((p+1)\), also multiply to 
$(N-1)(p+1) = (2p-1)(p+1) = 2p^2+p-1$.
\end{proof}
\begin{lemma}[The state is the remaining set]\label{lem:state}
Let $A\subseteq\{2,\dots,N\}$ be the set of labels not yet placed, and let $k=(N-1)-|A|$ be
the number already placed. Then the current exponent index is
\begin{equation}\label{eq:EA}
  E(A) \;=\; p \;+\; (p+1)\,|A| \;-\; \sum_{j\in A} j ,
\end{equation}
and the depth is $k=(N-1)-|A|$. Neither depends on the \emph{order} in which the labels of
$\{2,\dots,N\}\setminus A$ were placed.
\end{lemma}
 
\begin{proof}
Let $C=\{2,\dots,N\}\setminus A$ be the placed labels, $|C|=k$. By \eqref{eq:Ek},
\[
  E_k \;=\; \sum_{c\in C}(c-1) - (k-1)p \;=\; \Bigl(\sum_{c\in C}c\Bigr) - k - (k-1)p
  \;=\; \Bigl(\sum_{c\in C}c\Bigr) - k(p+1) + p .
\]
Writing $S=\sum_{j=2}^{N}j$ and $\sum_{c\in C}c = S-\sum_{j\in A}j$, and using
$k=(N-1)-|A|$,
\[
  E_k \;=\; S - \sum_{j\in A}j - (N-1)(p+1) + |A|(p+1) + p .
\]
By Lemma~\ref{lem:zerosum}, $S-(N-1)(p+1)=0$, which gives \eqref{eq:EA}. That $k$ is
determined by $|A|$ is immediate, since one label is placed at each step.
\end{proof}

\begin{rem}
The previous version of our algorithm carried three arguments: the tuple of
remaining labels, the running exponent, and the depth. Lemma~\ref{lem:state} says that the
second and third are functions of the first. They were never distinguishing two different
situations; they were being recomputed, stored and compared for nothing. In particular two
partial constructions that have placed the same labels in different orders are the same state
and must not be stored twice. 
\end{rem}

\begin{rem}
    Setting $A=\varnothing$ in \eqref{eq:EA} gives $E(\varnothing)=p$, which is the statement
$E_{N-1}(\sigma)=p$ for every $\sigma\in\Phi_p$.
\end{rem}

We %can 
now build %define 
the new recursive algorithm.

\begin{proposition}\label{prop:recursion}
Define the function $\Psi$ on subsets $A\subseteq\{2,\dots,N\}$ by $\Psi(\varnothing)=1$ and, recursively,
\begin{equation}\label{eq:psi}
  \Psi(A) \;=\; \sum_{\substack{e\in A\\ E(A\setminus\{e\})\,\geqslant\, p}}
  (-1)^{\#\{f\in A\,:\,f>e\}}\;\; E\bigl(A\setminus\{e\}\bigr)^{\underline{p}}\;\;
  \Psi\bigl(A\setminus\{e\}\bigr),
\end{equation}
with $E(\cdot)$ as in~\eqref{eq:EA}.
By definition, put $W \mathrel{{:}{=}} \prod_{k=0}^{N-1} k!$. 
Then
\[
  \Psi\bigl(\{2,\dots,N\}\bigr) \;=\; \sum_{\sigma\in\Phi_p}(-1)^{\sigma}\prod_{k=1}^{N-1}E_k(\sigma)^{\underline p}
  \;=\; \consta(p)\cdot W .
\]
\end{proposition}
 
\begin{proof}
Expanding \eqref{eq:psi} recursively produces exactly one term for each way of placing the
labels $\{2,\dots,N\}$ one at a time, that is, one term for each $\sigma\in S_N$ with
$\sigma(1)=1$. The constraint $E(A\setminus\{e\})\geqslant p$ imposed at every step is the
condition $E_k(\sigma)\geqslant p$ for all $k$, which by Section~\ref{sec:phip} is exactly the
condition $\sigma\in\Phi_p$; the surviving terms are therefore indexed by $\Phi_p$. By
Lemma~\ref{lem:state} the factor contributed after the $k$-th placement is
$E_k(\sigma)^{\underline p}$, so the product of the factors along a surviving branch is
$c(\sigma,p)$. The number \(\#\{f\in A\,:\,f>e\}\) %simply 
accounts for the inversions created by inserting element \(e\) to the left of not-yet-placed values. The stated identity is then Proposition~\ref{prop:closedform}.
\end{proof}

\begin{ex}[$p=2$, reproducing Example~\ref{ex:peq2} of Section~\ref{sec:oldalg}]\label{ex:p2}
Here $N=4$ and the labels are $\{2,3,4\}$, with $E(A)=2+3|A|-\sum_{j\in A}j$ by
\eqref{eq:EA}. The values of $\Psi$ are built from the empty set upwards:
\[
  \Psi(\varnothing)=1,\qquad E(\varnothing)=2 .
\]
\underline{Singletons.} $E(\{2\})=3$, $E(\{3\})=2$, $E(\{4\})=1$. For each singleton $\{e\}$ the only
move is to place $e$, leaving $\varnothing$ with $E=2\geqslant p$, contributing
$2^{\underline 2}=2$ with sign $+$:
\[
  \Psi(\{2\})=\Psi(\{3\})=\Psi(\{4\})=2 .
\]
\underline{Pairs.}
\[
\begin{aligned}
  \Psi(\{2,3\}) &= \underbrace{-\,2^{\underline 2}\,\Psi(\{3\})}_{\text{place }2,\ E(\{3\})=2}
                 + \underbrace{+\,3^{\underline 2}\,\Psi(\{2\})}_{\text{place }3,\ E(\{2\})=3}
                 = -2\cdot 2 + 6\cdot 2 = 8,\\
  \Psi(\{2,4\}) &= \underbrace{\text{place }2:\ E(\{4\})=1<p,\ \text{rejected}}_{}
                 \;+\; 3^{\underline 2}\,\Psi(\{2\}) = 6\cdot 2 = 12,\\
  \Psi(\{3,4\}) &= \underbrace{\text{place }3:\ E(\{4\})=1<p,\ \text{rejected}}_{}
                 \;+\; 2^{\underline 2}\,\Psi(\{3\}) = 2\cdot 2 = 4 .
\end{aligned}
\]
\underline{The full set.} $E(\{3,4\})=1<p$ rejects the move that places $2$ first, and
\begin{equation}\label{EqFull234}
  \Psi(\{2,3,4\}) \;=\; -\,2^{\underline 2}\,\Psi(\{2,4\}) \;+\; 3^{\underline 2}\,\Psi(\{2,3\})
  \;=\; -2\cdot 12 + 6\cdot 8 \;=\; 24 .
\end{equation}
The three surviving branches
are the three permutations $(1,2,3,4)$, $(1,2,4,3)$, $(1,3,2,4)$ of Example~\ref{ex:peq2}, and their
contributions $72$, $-24$, $-24$ appear in~\eqref{EqFull234} %here 
in the factored form:
\[
  \underbrace{6\cdot(6\cdot 2)}_{(1,2,3,4)}
  \;+\;\underbrace{6\cdot(-2\cdot 2)}_{(1,3,2,4)}
  \;+\;\underbrace{(-2)\cdot(6\cdot 2)}_{(1,2,4,3)}
  \;=\;6\cdot(12-4)\;-\;2\cdot 12\;=\;24 ,
\]
whence two permutations give us $6\cdot(12-4) =6\cdot 8$, and 
the third permutation produces $-2\cdot 12$ as stated.
Since $W=\prod_{k=0}^{3}k!=12$, this gives $\consta(2)=24/12=2$. 

\smallskip
Algorithmically, we start from
$A = \{2,3,4\}$ (nothing placed yet) and we see from Table~\ref{TabExp2A234} 
 %on p.~\pageref{TabExp2A234} 
that only six states are ever visited.\footnote{\label{FootP2yetSmall}
We admit that at this differential order ($p=2$), the table (six states)
is of course not yet smaller than $\Phi_2$ (three permutations); 
the gain appears as $p$ grows.}
%%%
\begin{table}[htb]
\caption{From the whole set $A=\{2,3,4\}$ not yet placed --- to all of the set $A$ gradually placed (right-to-left).}\label{TabExp2A234}
\[
\renewcommand{\arraystretch}{1.2}
\begin{array}{lll}
\hline\hline
A \text{ (not yet placed)} & \text{reached after placing} & \Psi(A) \\
\hline
\{2,3,4\} & \text{---} & 24 \\
\{2,4\} & 3 & 12 \\
\{2,3\} & 4 & 8 \\
\{2\} & 3, 4 \ \text{(in either order)} & 2 \\
\{3\} & 4, \text{ then } 2 & 2 \\
\varnothing & \text{everything} & 1 \\
\hline\hline
\end{array}
\]
\end{table}
%%%
The sharing is visible at the state $A = \{2\}$: it is reached both after
placing $3$ then $4$ and after placing $4$ then $3$, that is, the
permutations $(1,2,4,3)$ and $(1,2,3,4)$ finish through one and the same
table entry, which is computed once. 
%%%
To conclude the calculation of $\consta(p=2)$, the algorithm takes $\Psi\bigl( \{2,3,4\} \bigr) = 24$ from~\eqref{EqFull234} and divides it by the Wronskian $W = \prod_{k=0}^{4-1} k! = 0! 1! 2! 3! = 12$: this yields $\consta(2) = 2$.

\smallskip

Figure~\ref{FigStatesP2} shows these six states --\,see also Table~\ref{TabStates234} on p.~\pageref{TabStates234} below\,-- which are reached by the algorithm and which contribute to the sum in~\eqref{eq:psi}, and the factor carried by each step.\footnote{A step-by-step interactive walkthrough of this recursion, for all $p\leqslant 5$, is available at \url{https://kianshah.in/projects/constp}.}
\begin{figure}[htb]
\centering
\begin{tikzpicture}[
  st/.style={draw=teal!55!black, fill=teal!8, rounded corners=3pt,
    minimum width=17mm, minimum height=9mm, align=center, font=\small},
  sh/.style={st, line width=1.1pt, fill=teal!22},
  ar/.style={-{Stealth[length=2.2mm]}, draw=black!55},
  lb/.style={font=\scriptsize, fill=white, inner sep=1pt}]
\node[st] (a234) at (0,0)       {$\{2,3,4\}$\\ $\Psi=24$};
\node[st] (a23)  at (-2.3,-1.9) {$\{2,3\}$\\ $\Psi=8$};
\node[st] (a24)  at (2.3,-1.9)  {$\{2,4\}$\\ $\Psi=12$};
\node[st] (a3)   at (-2.3,-3.8) {$\{3\}$\\ $\Psi=2$};
\node[sh] (a2)   at (2.3,-3.8)  {$\{2\}$\\ $\Psi=2$};
\node[st] (a0)   at (0,-5.7)    {$\varnothing$\\ $\Psi=1$};
\draw[ar] (a234) -- node[lb]{$e=4\colon\ {+6}$} (a23);
\draw[ar] (a234) -- node[lb]{$e=3\colon\ {-2}$} (a24);
\draw[ar] (a23)  -- node[lb]{$e=2\colon\ {-2}$} (a3);
\draw[ar] (a23)  -- node[lb, pos=0.4]{$e=3\colon\ {+6}$} (a2);
\draw[ar] (a24)  -- node[lb]{$e=4\colon\ {+6}$} (a2);
\draw[ar] (a3)   -- node[lb]{$e=3\colon\ {+2}$} (a0);
\draw[ar] (a2)   -- node[lb]{$e=2\colon\ {+2}$} (a0);
\end{tikzpicture}
\caption{The six states $A$ of recursion~\eqref{eq:psi} at $p=2$ and their values
$\Psi(A)$. The arrow $A\to A\setminus\{e\}$ places $e$ and carries the factor
$(-1)^{\#\{f\in A\,:\,f>e\}}\allowbreak\,{E(A\setminus\{e\})^{\underline{2}}}$.}
\label{FigStatesP2}
\end{figure}
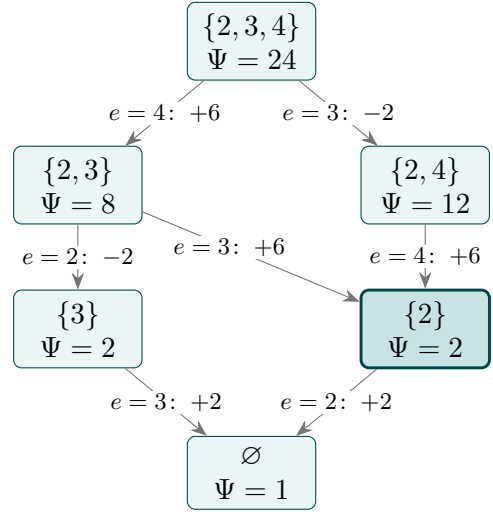
\end{ex}

%\noindent
\textbf{Storing permutation subsets as bit masks.}\quad 
In Section~\ref{sec:oldalg} the advantage of memoization was touched upon, which is now directly visible (e.g., from Table~\ref{TabExp2A234}): the same subset \(A\) appears repeatedly during our computation, and recognizing it as an already evaluated expression reduces computational time at the expense of memory. The number of states visited in any memoized algorithm remains the same, thus this is not where our subset reformulation is an advancement of the algorithm. 

The advancement lies instead in what a single stored state costs. Any memoized
algorithm must record, next to each stored value, some label saying which state
that value belongs to; without the label, the value could never be found again.
In the previous algorithm (Section~\ref{sec:oldalg}), this label is the list of weights not yet
placed, together with the running exponent and the depth, an object whose size
grows with $p$: at $p = 11$ it occupies $188$ bytes against the $94$ bytes of
the value it identifies, so that more than half of the table is spent on
addressing rather than on content. Lemma~\ref{lem:state} removes this. Since the
state is nothing but a subset of $\{2, \ldots, N\}$, it is recorded by $2p-1$
bits, that is, by one machine word %integer 
for every $p \leqslant 32$, and the label
becomes both smaller than the value it labels and independent of~$p$.

\begin{implementNo}
Eeach state $A$ is written as its bit mask: the presence of a label is encoded by the bit~$1$, the absence by~$0$. 
The bits for (non)inclusion of each label $N$,\ $N-1$,\ $\dots$,\ $2$ are ordered from
left to right. 

For example, the subset \( \{2\} \subsetneq \{2,3,4\} \) is encoded by the bit mask \(001\) w.r.t.\ the label ordering $4\prec 3 \prec 2$.
 %for the set is since we include (include \(\implies 1\)) the `last' element. 
\end{implementNo}

\begin{rem}\label{rem:states}
The number of states actually stored by %the 
recursion~\eqref{eq:psi} itself is
an integer sequence: for $p = 1, \ldots, 8$ it reads
\[
  2,\quad 6,\quad 20,\quad 74,\quad 282,\quad 1100,\quad 4332,\quad 17144,
\]
which are the terms of odd indices $2p-1$ in the integer sequence \seqnum{A222855} 
(see~\cite{oeisA222855}).\footnote{\label{FootBitMaskIncline}
The definition of this integer sequence~\seqnum{A222855} refers to the bit mask representation of subsets $A\subseteq\{1$,\ $\ldots$,\ $n\}$.
Specifically, the states that are counted are exactly those whose bits, 
plotted from right to left as points in $\BBZ\times\{0,1\} \subsetneq \BBZ^2$,
  %as points from left to right, 
have the least\/-\/squares line interpolation that does not slope down:
e.g., $001$ (growth) or $010$ (horizontal line) but not~$100$ (decline).%
} %; such bit strings are what \seqnum{A222855} counts. 
%%%
\end{rem}

\begin{ex}[$p=2$]
By listing all the subsets of $\{2,3,4\}$ in Table~\ref{TabStates234}, let us compare which of them contribute to the sum $\Psi\bigl( \{2,3,4\} \bigr)$ in the recursive formula~\eqref{eq:psi} --- and which of their bit masks are counted by the integer sequence~\seqnum{A222855}.
%%%
\begin{table}[htb]
\caption{The subsets of $\{2,3,4\}$, their bit masks, and the slopes of linear interpolations (if nonnegative, then counted by the integer sequence~\seqnum{A222855}.}\label{TabStates234}
\[
\renewcommand{\arraystretch}{1.15}
\begin{array}{l|cccccccc}
\hline\hline
A & \varnothing & \{2\} & \{3\} & \{4\} & \{2,3\} & \{2,4\} & \{3,4\} & \{2,3,4\}\\[1pt]
\hline
\text{bits for } 4,3,2 & 000 & 001 & 010 & 100 & 011 & 101 & 110 & 111\\
\text{slope} & 0 & + & 0 & - & + & 0 & - & 0\\
\text{state} & \checkmark & \checkmark & \checkmark & & \checkmark & \checkmark & & \checkmark \\
\hline\hline
\end{array}
\]
\end{table}
%%%
The two lists match! This ought to be so by Remark~\ref{rem:idrev} on p.~\pageref{rem:idrev}. Indeed, it is the high\/-\/value labels (close to $N$) which are preferably taken first, hence these are likely absent from the subsets $A\subseteq \{2,\ldots,N\}$ of labels that remain available. The lower values (close to $2$) are typically picked late, hence these labels stay in the subsets $A$ for many iterations of the recursive algorithm in formula~\eqref{eq:psi}. The absent, already\/-\/taken high\/-\/value labels are encoded by bits~$0$ at the left end of the bit mask for the still\/-\/available label set~$A$; the remaining low\/-\/value labels are encoded by the bits~$1$ near the right end of the mask. Hence the nonnegative incline of the interpolating straight line.
\end{ex}

\smallskip
%\noindent
\textbf{Evaluation modulo primes.}\quad 
As for our final improvement, we note that the entries $\Psi(A)$ are large. They must be, since $\Psi(\{2,\dots,N\})
=\consta(p)\cdot W$ and the monomials' Wronskian $W$ alone has $1274$ binary digits at $p=15$. Measured at $p=12$, the
average entry already has $644$ binary digits. This, and not the number of entries, is what
exhausts the memory. The remedy is that formula~\eqref{eq:psi} is built from addition, subtraction and multiplication only.
 
\begin{define}\label{DefPsiQ}
Let $q > 2p$ be a prime, and let $\pi_q\colon \mathbb{Z}\to\mathbb{Z}/q\mathbb{Z}$ be the quotient map.
Define the $\BBZ / q\BBZ$\/-\/valued function $\Psi_q$ by %the 
recursion~\eqref{eq:psi} with every occurrence of $+$, $-$ and $\cdot$
read in $\mathbb{Z}/q\mathbb{Z}$ modulo the prime~$q$: % > 2p
postulate $\Psi_q(\varnothing)=\pi_q(1)$ and replace 
the falling factorials by their images $\pi_q\bigl(E^{\underline p}\bigr)$:
\begin{equation}\label{eq:psiq}\tag{\ref*{eq:psi}$'$}
  \Psi_q(A) \;=\; \sum_{\substack{e\in A\\ E(A\setminus\{e\})\,\geqslant\, p}}
  (-1)^{\#\{f\in A\,:\,f>e\}}\;\;
  \pi_q\bigl(E(A\setminus\{e\})^{\underline{p}}\bigr)\;\;
  \Psi_q\bigl(A\setminus\{e\}\bigr).
\end{equation}
\end{define} 
 
\begin{lemma}\label{lem:hom}
Under the above assumptions and notation, %of the above definition, 
we have that the $\BBZ$-\/valued function $\Psi$ from~\eqref{eq:psi} satisfies the equality
%Then
$\Psi_q(A)=\pi_q\bigl(\Psi(A)\bigr)$ for every subset~$A\subseteq \{2$,\ $\ldots$,\ $N\}$ yet to be places.
\end{lemma}

\begin{proof}
Induction on $|A|$. 
The base case $A=\{2,\ldots,N\}\setminus \{e\}$ is the definition. 
The condition $E(A\setminus\{e\})\geqslant p$
selecting which terms occur is a condition on the integers $E(\cdot)$, computed by
\eqref{eq:EA} in $\mathbb{Z}$, and is therefore unchanged. Each surviving term is a product of
three integers, and $\pi_q$ is a ring homomorphism, so the image of the sum of the terms is the
sum of the images.
\end{proof}

Lemma~\ref{lem:hom} gives us the same recursion over
$\mathbb{Z}/q\mathbb{Z}$ instead of $\mathbb{Z}$. Every stored entry is now an element of
$\mathbb{Z}/q\mathbb{Z}$, which for $q<2^{62}$ is one machine integer. The falling factorials
$E^{\underline p}$ are tabulated once per prime for $p\leqslant E\leqslant p(p+1)/2$, which is the full
range by Lemma~8.2 of~\cite{ShahKiselev2026}. Every time the recursion places a weight it needs the integer $E^{\underline p} = E(E-1)\cdots(E-p+1)$, i.e.\ a product of $p$ terms, and computing it %this 
from scratch each time costs $p$ multiplications. Thus, we just store all possible values of $E$, which is at most $p(p+1)/2 - p + 1$ distinct values, modulo our chosen primes. What remains is to divide by the monomials' Wronskian \(W\), which is also better done inside $\mathbb{Z}/q\mathbb{Z}$ than afterwards.

\begin{lemma}\label{lem:Winv}
Let $q>N$ be prime. Then the Wronskian determinant $W=\prod_{k=0}^{N-1}k!$ is invertible modulo $q$, and
\[
  \consta(p) \;\equiv\; \Psi_q\bigl(\{2,\dots,N\}\bigr)\cdot W^{-1} \pmod{q}.
\]
\end{lemma}
 
\begin{proof}
By construction, the value $W$ is the %a 
product of factorials $k!$ with $k\le N-1<q$, so no factor of $W$ is divisible by
$q$; thus $q\nmid W$ and $W$ is a unit ($\not\equiv 0$) modulo $q$. 
By Proposition~\ref{prop:recursion} and
Lemma~\ref{lem:hom}, $\Psi_q(\{2,\dots,N\})=\pi_q(\consta(p)\cdot W)$, hence %and 
multiplying by $W^{-1}$ asserts %yields %gives 
the claim.
\end{proof}

\begin{algorithm}[htb]
\caption{The deployed evaluation of $\consta(p)$: %the 
recursion~\eqref{eq:psi} swept by cardinality, modulo primes.}
\label{alg:deployed}
\begin{algorithmic}[1]
\Require $p \geqslant 1$; put $N := 2p$ and $W := \prod_{k=0}^{N-1} k!$
\Require pairwise distinct primes $q_1, \ldots, q_r$, each $q_i > N$, with
  $q_1 q_2 \cdots q_r > M_p$
\For{$i = 1, \ldots, r$ \textbf{independently, in parallel}}
  \State $q \gets q_i$
  \For{$E = p, \ldots, p(p+1)/2$}
    \State $F[E] \gets E^{\underline{p}} \bmod q$
      \Comment{tabulated once; this is the full range by
      Lemma~8.2 of \cite{ShahKiselev2026}}
  \EndFor
  \State $\Psi_q[\varnothing] \gets 1$
  \For{$m = 1, \ldots, N-1$}
    \ForAll{$A \subseteq \{2, \ldots, N\}$ with $|A| = m$}
      \State $\Psi_q[A] \;\gets\;
        \sum\limits_{\substack{e \in A \\ E(A \setminus \{e\}) \,\geqslant\, p}}
        (-1)^{\#\{f \in A \,:\, f > e\}}\;
        F\bigl[E(A \setminus \{e\})\bigr]\;
        \Psi_q\bigl[A \setminus \{e\}\bigr] \;\bmod q$
        \Comment{the recursion \eqref{eq:psi}, with $E(\cdot)$ from
        \eqref{eq:EA}}
    \EndFor
    \State discard the stored layer $m-1$
      \Comment{by \eqref{eq:psi} only sets of size $m-1$ are ever read}
  \EndFor
  \State $\rho_i \gets \Psi_q\bigl[\{2, \ldots, N\}\bigr] \cdot W^{-1}
    \bmod q$
    \Comment{the residue of $\consta(p)$, by Lemma~\ref{lem:Winv}}
\EndFor
\State \Return the unique integer $C$ with $0 \leqslant C < q_1 \cdots q_r$ and
  $C \equiv \rho_i \pmod{q_i}$ for every $i$
  \Comment{Chinese remainder theorem; $C = \consta(p)$ since
  $0 < \consta(p) \leqslant M_p < q_1 \cdots q_r$}
\end{algorithmic}
\end{algorithm}

We run Algorithm~\ref{alg:deployed} for
several primes $q_1, \ldots, q_r$. The Chinese Remainder Theorem guarantees that a nonnegative integer smaller than the product $q_1 \cdots q_r$ is
uniquely determined by its residues modulo the pairwise distinct primes
$q_i$. Thus, we are able to reconstruct the exact integer $\consta(p)$; here
$\consta(p) > 0$. The number of
primes to take is known in advance: for every $\sigma \in \Phi_p$ and
every $k$, the sum of the $k$ values placed last is at most the sum of
the $k$ largest values, which is attained by the identity permutation, so
$E_k(\sigma) \leqslant E_k(\sigma_{\mathrm{id}})$ and thus
\[
  \consta(p)
  \;\leqslant\;
  (N-1)!\;\prod_{k=1}^{N-1}
  E_k(\sigma_{\mathrm{id}})^{\,\underline{p}} \,\Big/\, W
  \;\mathrel{{=}{:}}\; M_p.
\]

At $p = 15$ the bound $M_p$ occupies $1470$ binary digits, so $r = 25$
primes just above $2^{61}$ suffice; the value $\consta(15)$ itself
turns out to occupy $960$ bits, i.e.\ $16$ such primes. The primes are
processed completely independently of one another, which is where the
parallelism lives.

Table~\ref{tab:reach} reports the resulting resource usage. With these
choices, the value $\consta(15)$, an integer of $289$ decimal
digits, takes about twelve minutes and $1.2$ GB of memory on two laptop
cores; a $128$ GB computer node reaches $p = 18$.
%%%
\begin{table}[htb]
\caption{Measured %and estimated 
resource usage of the layered,
multi\/-\/modular implementation of %the 
recursion~\eqref{eq:psi}.}
\label{tab:reach}
\centering
\renewcommand{\arraystretch}{1.25}
\begin{tabular}{crr}
\hline\hline
$p$ & peak memory & computation time \\
\hline
$14$ & $0.30$ GB & $204$ s \\
$15$ & $1.16$ GB & $704$ s \\
$16$ & $4.52$ GB & $3{,}357$ s
\\
$17$ & $17.12$ GB 
& $11{,}215$ s
\\
$18$ & $68.66$ GB  
& $45{,}759$ s 
\\
\hline\hline
\end{tabular}
\end{table}

%We have reduced the problem of finding the values of $\consta(p)$
%in \eqref{eq:main} to the evaluation of formula \eqref{eq:closedform},
%carried out by the memoized recursion of Section~\ref{sec:memo}. 
%The resulting values are summarized in Table~\ref{tab:values}; the exact
%decimal expansions and the prime factorizations are recorded in
%Appendix~\ref{app:data}. 

%%%
\begin{table}[htb]
\caption{The values of $\consta(p)$ for $p \leqslant 18$: exact values
for $p \leqslant 6$ and rounded approximations beyond, with the number of
decimal digits.
  % (predicted for $19\leqslant p \leqslant 22$ by using
  %the data from~\S\ref{sec:growth}).
   %see also Claim~\ref{ClaimAsympt} on p.~\pageref{ClaimAsympt}). 
Exact values for all $p \leqslant 18$ appear in Appendix~\ref{app:data}.}
\label{tab:values}
\centering
\renewcommand{\arraystretch}{1.25}
\begin{tabular}{ccl}
\hline\hline
$p$ & digits & $\consta(p)$ \\
\hline
$1$ & $1$ & $1$ \\
$2$ & $1$ & $2$ \\
$3$ & $2$ & $90$ \\
$4$ & $6$ & $586656$ \\
$5$ & $13$ & $1915103977500$ \\
$6$ & $22$ & \seqsplit{7886133184567796056800} \\
$7$ & $35$ & $\approx 8.587 \times 10^{34}$ \\
$8$ & $52$ & $\approx 4.594 \times 10^{51}$ \\
$9$ & $73$ & $\approx 2.059 \times 10^{72}$ \\
$10$ & $98$ & $\approx 1.236 \times 10^{97}$ \\
$11$ & $127$ & $\approx 1.512 \times 10^{126}$ \\
$12$ & $160$ & $\approx 5.498 \times 10^{159}$ \\
$13$ & $198$ & $\approx 8.396 \times 10^{197}$ \\
$14$ & $241$ & $\approx 7.398 \times 10^{240}$ \\
$15$ & $289$ & $\approx 5.046 \times 10^{288}$ \\
$16$ & $342$ & $\approx 3.503 \times 10^{341}$ \\
$17$ & $400$ & $\approx 3.198 \times 10^{399}$ \\
$18$ & $463$ & $\approx 4.881 \times 10^{462}$ \\
\hline\hline
\end{tabular}
\end{table}
%%%

\section{New values of the new sequence, and its growth rate}\label{sec:results}
%Numerical results

\noindent%
By running Algorithm~\ref{alg:deployed} described in the preceding section, we obtained the exact values of $\consta(p)$ for all $p\leqslant 18$.
Their decimal orders of magnitude are presented in Table~\ref{tab:values};
the integer numbers themselves --\,in decimal notation\,-- are found in Appendix~\ref{app:data}.
%%%(Table)-->[above Sec.]
By reading the values $\consta(p)$, we discover that they form a previously unknown sequence: it is new\,!

The positive integer sequence $\bigl\{\consta(p)\bigr\}_{p\in\BBN}$ is %new,
   %The integer sequence $\consta(p)$, $p \in \mathbb{Z}_{\geqslant 1}$, 
now recorded in the {\it On-Line Encyclopedia
of Integer Sequences} \cite{oeisA392714} as~\seqnum{A392714}.

The observed growth rate is $\consta(p)\sim p^{2 p^2}$, 
as seen from the interpolation of $\log\consta(p)$ which we discuss in~\S\ref{sec:growth} below %Claim~\ref{ClaimAsympt} on p.~\pageref{ClaimAsympt} below 
(it provides the expected number of decimal digits for the unknown values).
The super\/-\/exponential rapidity of growth explains why --\,by using the simple combinatorial Algorithm~\ref{alg:old} 
(implemented in Appendix~\ref{app:oldcode}, cf.~\cite{SciPost})\,--
it becomes so costly to accumulate the values $\consta(p)$ at $p\geqslant 10$ by actually counting \emph{every} incremental contribution from permutations $\sigma\in S_{2p}$ to these huge numbers${}\sim 10^{n\cdot 100}$.

\subsection{%\texorpdfstring{$p$}{p}-adic valuations
Prime factorisations of known $\consta(p)$}\label{sec:padic}

Let us analyse the prime decompositions %factorisations 
of positive integers $\consta(p)$. In Appendix~\ref{app:data} we factor each value $\consta(p)$ into a product of integers which are either proven prime or too big, so that primality tests only tell us that those numbers are composite (without revealing their prime factorisation).\footnote{\label{FootPrimesInsideOut}
We foresee that our computation of the integer values $\consta(p) \sim 10^{n\cdot 100}$ will be used in future to generate huge primes of the same order.
Indeed, the factorisation of $\consta(p=14)$ gave us the 206-\/decimal\/-\/digit number, by now confirmed to be a new prime (see App.~\ref{app:data}).}

It is readily seen from the data in Appendix~\ref{app:data} that the prime decomposition of each known $\consta(p)$ consists of a nucleus of small primes in high powers, times few big primes in power~$1$.
%%%
\begin{table}[htb]
\caption{The $q$-adic valuations $v_q(\consta(p))$ for primes
$q \leqslant 19$ and $2 \leqslant p \leqslant 18$. The entries $v_p(\consta(p))$
for prime $p$ are in boldface.}
\label{tab:valuations}
\centering
\normalsize%\small%\footnotesize
\setlength{\tabcolsep}{3.1pt}
\renewcommand{\arraystretch}{1.2}
\begin{tabular}{r*{17}{c}}
\hline\hline
$q \,\backslash\, p$
  & 2 & 3 & 4 & 5 & 6 & 7 & 8 & 9 & 10 & 11 & 12 & 13 & 14 & 15 & 16 & 17 & 18 \\
\hline
 2 & $\mathbf{1}$ & 1 & 5 & 2 & 5 & 4 & 17 & 11 & 16 & 7 & 18 & 13 & 17 & 11 & 49 & 33 & 32\\
 3 & $\cdot$ & $\mathbf{2}$ & 3 & 1 & 6 & 2 & 2 & 14 & 6 & 5 & 11 & 8 & 3 & 12 & 3 & 3 & 31\\
 5 & $\cdot$ & 1 & $\cdot$ & $\mathbf{4}$ & 2 & $\cdot$ & $\cdot$ & $\cdot$ & 8 & 6 & 5 & 2 & 2 & 12 & 6 & 1 & 2\\
 7 & $\cdot$ & $\cdot$ & 1 & 1 & $\cdot$ & $\mathbf{6}$ & 5 & 2 & $\cdot$ & $\cdot$ & $\cdot$ & $\cdot$ & 13 & 8 & 4 & 2 & $\cdot$\\
11 & $\cdot$ & $\cdot$ & $\cdot$ & $\cdot$ & 1 & $\cdot$ & $\cdot$ & 2 & $\cdot$ & $\mathbf{10}$ & 8 & 6 & 4 & 3 & 1 & $\cdot$ & $\cdot$\\
13 & $\cdot$ & $\cdot$ & $\cdot$ & $\cdot$ & $\cdot$ & 1 & 1 & 1 & $\cdot$ & 1 & $\cdot$ & $\mathbf{12}$ & 11 & 8 & 6 & 5 & 2\\
17 & $\cdot$ & $\cdot$ & $\cdot$ & $\cdot$ & $\cdot$ & $\cdot$ & $\cdot$ & 1 & $\cdot$ & $\cdot$ & $\cdot$ & 1 & $\cdot$ & 1 & $\cdot$ & $\mathbf{16}$ & 15\\
19 & $\cdot$ & $\cdot$ & $\cdot$ & $\cdot$ & $\cdot$ & $\cdot$ & $\cdot$ & $\cdot$ & 1 & 1 & $\cdot$ & $\cdot$ & $\cdot$ & $\cdot$ & $\cdot$ & 1 & $\cdot$\\
\hline\hline
\end{tabular}
\end{table}
%%%
For small primes $q\in\{ 2,3,5,7,11,17,19\}$ and $2 \leqslant p \leqslant 18$, we display in Table~\ref{tab:valuations} %on p.~\pageref{tab:valuations}
the powers $v_q(\consta(p))$ in which these primes~$q$ occur in the factorisation of~$\consta(p)$; a dot denotes %the
zero entry. (\textbf{NB:} Greater prime numbers, $q\geqslant 23$, do also contribute as $p$ approaches~$12$.)
   %In Table~\ref{tab:valuations} we display the $q$-adic valuations
   %$v_q(\consta(p))$ for all primes $q \leqslant 19$ and $2 \leqslant p 
   %\leqslant 14$, where a dot denotes a zero entry.
We detect that
   %One feature of Table~\ref{tab:valuations} is immediately apparent: 
the prime factors of a composite $p$ appear in $\consta(p)$ with notably high
valuations. The dominant contributions to the column of $q$-adic valuations $v_q(\consta(p))$ arise precisely when $q$ %\mid 
divides~$p$. In particular, for every prime $p\leqslant 17$ in our
range, the entry $v_p(\consta(p))$ equals exactly $p - 1$.

\begin{conjecture}\label{conj:valuation}
For every prime $p$, the $p$-adic valuation of $\consta(p)$~is 
\[
  v_p\bigl(\consta(p)\bigr) \;=\; p - 1.
\]
In~\cite{ShahKiselev2026} we establish %prove 
the lower bound $v_p(\consta(p))\geqslant p-1$, yet we already see that it actually is exact. We shall report on our progress in the subsequent paper~\cite{SciPost}.
\end{conjecture}
   %%% AI did not see any other laws or patterns.

\subsection{Growth rate}
\label{sec:growth}

By a simple count of decimal digits in the integer values $\consta(p)$ we deduce that this sequence grows faster than $p!$ or the factorial $(2p)!$, the number of permutations in the group $S_{2p}$ over which the differential operators of order~$p$ are alternated in formula~\eqref{eq:opdef}.
The growth rate of the number of digits in~$\consta(p)$ reveals %suggests 
--\,up to a multiplicative constant\,-- the growth rate of its decimal logarithm:
\[
\lg \consta(p) \sim p^2\cdot \lg p\qquad \text{as $p \gg 1 $.}
\]
Making preference of the \emph{natural} logarithms, $\log \consta(p)$ and $\log p$,
we introduce the Ansatz (currently neglecting any negative powers of~$p$ or $\log p$),
\begin{equation}\label{EqAnsatz}
\log \consta(p) \approx \alpha p^2\log p + \beta p^2 + \gamma p \log p + \delta p + \veps \log p + \mu;
\end{equation}
the parameters $\alpha,\beta,\gamma,\delta,\veps,\mu$ are determined from the (sub)sets of known values for $\log \consta(p)$ by the least squares minimisation, %Gauss
e.g., over the segments $[p_{\min}\gg 1, p_{\max} = 18]$.
%%%
The best\/-\/fit values from the Gauss method\footnote{\label{FootTooFewTooMany}
Whenever the left end of the segment $[p_{\min}, p_{\max}]$ is too close to~$1$, the input data are distorted by the contribution of sub\/-\/leading terms in the asymptotic approximations like the Stirling formula $p! \approx p^p\cdot \exp(-p)$ for~$p\gg 1$.
But whenever the left end of the segment is too close to its right end, the amount of input data becomes insufficient to fix the parameter values, whence their instability.} 
%%%
are shown in Table~\ref{TabBestFit}.
%%%
\begin{table}[htb]
\caption{The best\/-\/fit values of the parameters $\alpha,\beta,\gamma,\delta,\veps,\mu$
within the Ansatz in~\eqref{EqAnsatz}, 
over six fitting segments $[p_{\min}, p_{\max}]$.}\label{TabBestFit}
\begin{center}{\footnotesize%
\begin{tabular}{rccccccl}
\hline\hline
segment\mbox{ } & $\alpha$ & $\beta$ & $\gamma$ & $\delta$ & $\varepsilon$ & $\mu$ & max$|$resid$|$\\
\hline
$[ 5\ldots 18]$  & $1.998939$ & $-2.669619$ & $0.9239497$ & $-0.05930426$ & $3.05443\phantom{?}$ &    $2.537441$ & $3.59\cdot 10^{-6}$ \\
$[ 6\ldots 18]$ & $1.999124$ & $-2.670625$ & $0.9322903$ & $-0.07935707$ & $3.077092$ &    $2.551782$ & $1.2\phantom{?}\cdot 10^{-6}$ \\
$[ 7\ldots 18]$ & $1.999247$ & $-2.6713\phantom{??}$ & $0.9381615$ & $-0.09384501$ & $3.094169$ &   $ 2.561343$ & $4.21\cdot 10^{-7}$ \\
$[ 8\ldots 18]$ & $1.999338$ & $-2.671801$ & $0.9427179$ & $-0.1053434$ & $3.108234$ &    $2.568331$ & $1.61\cdot 10^{-7}$ \\
$[ 9\ldots 18]$ & $1.999409$ & $-2.672197$ & $0.9464585$ & $-0.1149704$ & $3.120406$ &    $2.573704$ & $6.39\cdot 10^{-8}$ \\
$[10\ldots 18]$ & $1.999466$ & $-2.672521$ & $0.9496362$ & $-0.1232932$ & $3.131249$ &    $2.577952$ & $2.26\cdot 10^{-8}$ \\
\hline\hline
\end{tabular}}%\footnotesize
\end{center}
\end{table}

Let us emphasize that by construction, these parameter values ($\alpha\approx 2$, 
$\beta\approx -2.67$, etc.) are {not} obtained from a study of the \emph{asymptotic} behaviour of $\log\consta(p)$ as~$p\to +\infty$.
All the more surprising it is that over $p\gtrsim 10$, the best\/-\/fit parameter values approach the known description of the asymptotic formula of $\log \consta(p)$ as~$p\to +\infty$.

\begin{claim}[{O.\,Zaboronski${}^{\text{\protect\ref{FootZabor}}}$}]\label{ClaimAsympt}
As $p\to +\infty$, the asymptotic expansion is
\begin{equation}\label{EqAsymptZabor}
\log\consta(p) = 2p^2\log p - b p^2 + \overline{o}(p^2),
\end{equation}
where $b\geqslant 2.6744$ (meaning $\beta \leqslant -2.6744$ in~\eqref{EqAnsatz}).
\end{claim}

Both the Ansatz in Eq.~\eqref{EqAnsatz} (with $\alpha\approx 2$ and $\beta\approx -2.67$ in Table~\ref{TabBestFit}) and the asymptotic formula in Eq.~\eqref{EqAsymptZabor} (with $\alpha=2$ and $\beta\leqslant -2.6744$) yield good estimates for the number of decimal digits in the yet unknown values, e.g., $\consta(p=19\ldots 22)$, see Table~\ref{TabPredictNumDigits}. The expected order of magnitude determines the postulated number of large ($\geqslant 2p$) primes in the Chinese Remainder Theorem (which we use to speed up Algorithm~\ref{alg:deployed} by avoiding summation of huge integers${}\sim 10^{n\cdot 100}$).
%%%
\begin{table}[htb]
\caption{Extrapolated values $\log\consta(p\geqslant 19)$ from Ansatz~\eqref{EqAnsatz} over segments $[p_{\min},p_{\max}=18]$, the expected number of decimal digits in the not yet known integers~$\consta(p\geqslant 19)$, and their estimates.}\label{TabPredictNumDigits}
\begin{center}
\begin{tabular}{cccccc}
\hline\hline
$p$ & $\log\consta(p)$ & $\log\consta(p)$ & $\log\consta(p)$ & {\#}digits in $\consta(p)$ & $\consta(p)$ \\
    & $[8\ldots18]$ & $[9\ldots18]$ & $[10\ldots18]$ & (expected) & (predicted) \\
\hline
$19$ & $1223.1197$ & $1223.1198$ & $1223.1192$ & $532$  & $\approx 1.563\cdot 10^{531}$ \\
$20$ & $1393.3277$ & $1393.3278$ & $1393.3271$ & $606$  & $\approx 1.301\cdot 10^{605}$\\
$21$ & $1576.2076$ & $1576.2077$ & $1576.2070$ & $685$  & $\approx 3.45 \cdot 10^{684}$\\
$22$ & $1771.9531$ & $1771.9532$ & $1771.9524$ & $770$  & $\approx 3.54 \cdot 10^{769}$\\
\hline\hline
\end{tabular}
\end{center}
\end{table}

\section{Conclusion}\label{sec:conclusion}
\noindent%
Naturally appearing in Eq.~\eqref{eq:main} from an $N$-ary generalisation for the construction of binary Lie bracket $[{\cdot},{\cdot}]$ for vector fields $\vec{w}_j\in\mathrm{D}_{p=1}(\BBR^1)$ on the real line, the positive integer sequence $\consta(p)$ is new. This is established by comparing its values at $p\leqslant 18$ (in Table~\ref{tab:values} on p.~\pageref{tab:values} and in Appendix~\ref{app:data} below) with all other entries in the On-Line Encyclopedia of Integer Sequences (OEIS). 
The new sequence, now registered as OEIS A392\,714, belongs to the class of combinatorial models that shows the asymptotic behaviour
\[
\log\consta(p) = \alpha p^2\log p + \beta p^2 +\cdots,
\]
with the approximate values $\alpha\approx 2$ and $\beta\approx -2.67$ observed by interpolating $\log\consta(p)$ on the segments $[6,18]$,\ $\ldots$,\ $[10,18]$ using the Ansatz~\eqref{EqAnsatz}.
Using techniques and methods for large deviation models in extremal combinatorics,
Zaboronski claims the exact value $\alpha = 2$ and provides the bound $\beta\leqslant -2.6744$ for the sub\/-\/leading coefficient (see footnote~\ref{FootZabor} on p.~\pageref{FootZabor} for reference).%1.SEP.2026

The algorithm to compute $\consta(p)$ at $p\gtrsim 10$ --\,which we design in this paper\,-- is implemented in \textsf{Python} and \textsf{C${+}{+}$};
these program code files are freely available from~\cite{ShahCode2026}.
  %and from the ancillary files catalogue accompanying this manuscript on arXiv.
Running these scripts on a powerful computing machine will unravel the so far unknown values $\consta(p\geqslant 19)$; the number of decimal digits in the integers $\consta(p)$ is predicted for $19\leqslant p \leqslant 22$ in Table~\ref{TabPredictNumDigits}.

%Our sequence $\consta(p)$ is natural, and it is new. It grows fast:
%on the available data ($1 \leqslant p \leqslant 15$) we see
%\[
%  \log \consta(p) \;\sim\; \alpha\, p^2 \log p + \beta\, p^2
%  + \cdots,
%  \qquad \text{where } \alpha \approx 2,\ \beta \approx -2.86;
%\]
%this will be presented in more detail in \cite{ShahKiselev2026}.
%O.~Zaboronski studies the asymptotics as $p \to +\infty$: he claims
%$\alpha \equiv 2$, so that our match is perfect, and he gives a bound for
%$\beta$, using other techniques and methods.%\cite{Zaboronski}. 
%Our technique uses the set $\Phi_p$; there is more to know about the structure
%of this set (cf.\ the odd\slash even structure in Remark~\ref{obs:parity}).

\subsection*{Discussion}\label{sec:discussion}
The prime decomposition of known values $\consta(p)$ is very peculiar (see Table~\ref{tab:valuations} on p.~\pageref{tab:valuations} and Appendix~\ref{app:data} from p.~\pageref{app:data} onwards): every integer $\consta(p)\in\BBN$ factors into the nucleus of small primes ($< p$) in high powers and prime(s) between $p+1$ and $2p$, times very few very big primes ($\gg p$).
In our subsequent work~\cite{SciPost} we explore the $q$-adic valuations of $\consta(p)$ in more detail. In particular, we prove that for prime $p\in\BBN$, the power $p^{p-1}$ divides $\consta(p)$; our Conjecture~\ref{conj:valuation} suggests --\,on the basis of empiric evidence from $\consta(p)$ for all $p\leqslant 17$\,-- that the $p$-adic valuation of $\consta(p)$ always \emph{equals} $p-1$ for $p$ prime. The proof technique refers to the mathematics of the old Algorithm~\ref{alg:old} that computed $\consta(p)$ using the permutation group algebra.

\smallskip
The arithmetic of falling factorials in section~\ref{sec:phip}, see formula~\eqref{eq:closedform} in Proposition~\ref{prop:closedform} on p.~\pageref{prop:closedform}, dictates that only the relatively small subset $\Phi_p
\subseteq S_{2p-1} \subsetneq S_{2p}$ of late\/-\/growing permutations contributes to the value $\consta(p)$; this is valid under our initial assumption~\eqref{EqMonomWeight} which unlocks $p\geqslant 4$ after Table~\ref{tab:intro} in~\S\ref{sec:prelim}.
In the course of computing $\consta(p)$ that followed for $4\leqslant p \leqslant 14$ (see Table~\ref{tab:phip} on p.~\pageref{tab:phip}) we observed that at every~$p$, the set $\Phi_p$ consists of even and odd permutations in such a way that the difference of their numbers equals $(-1)^{p+1}$.
We claim this is true for all $p\in\BBN$: the filter (which discards the non\/-\/contributing permutations, often before fully constructing a given one) is read as a lattice walk, and two sign\/-\/reversing involutions on the orderings cancel almost everything (within $\Phi_p$) except a single permutation that alternates between even and odd parity. %Details will be provided elsewhere.
The proof is found in our next paper~\cite{OddEven}.

%\begin{rem}
%Although ``obsolete'', our method to get $\consta(p)$ is a rare
%example of effective computations in the permutation group algebra. 
%In \cite{ShahKiselev2026} we review that line and, using it, we prove that
%$p^{\,p-1}$ divides $\consta(p)$ for every prime $p$
%\cite[Theorem~8.6]{ShahKiselev2026}.
%The task is to prove the equality $v_p(\consta(p)) = p - 1$
%(Conjecture~\ref{conj:valuation}), which we already see on the data.
%\end{rem}

\subsection*{Open problems and perspectives}\label{sec:openproblems}
Finally, let us outline several open problems for future research;
we do not repeat here the questions which were posed 
in the preceding section.\footnote{\label{FootHugePrimeComp}
E.g., what are the prime factorisations of huge composite numbers ${}\sim 10^{n\cdot 100}$ appearing in the decompositions of $\consta(p)$ for $p\gtrsim 13$ (see Appendix~\ref{app:data})\,?
What is the law of prime factorisation for $\consta(p)$, seen at work in Table~\ref{tab:valuations}\,? (The answer to the second question will solve the former.)%
   %are the huge numbers${}\sim 10^{n\cdot 100}$ appearing in the factorisations of 
   %$\consta(p)$ for $p\gtrsim 10$, see Appendix~\ref{app:data}, truly prime or composite\,?
}

\begin{open}\label{OpenAsymptSeries}
The asymptotic series,
\[
\log\consta(p) \simeq 2p^2\log p -b p^2 + \gamma p\log p + \delta p + \veps\log p +\mu + \dots,
\]
likely expands w.r.t.\ \emph{negative} powers of the new variable $\log p$ as $p\to +\infty$. Can their coefficients be constrained by inequalities or algebraic equations\,?

In the limit $1/p \to +0$, does the asymptotic curve of $\consta(p)$ satisfy any differential equation --- stemming from the expansion of the Wronskian determinant 
of size $2(p+1) \times 2(p+1)$ near $\consta(p+1)$ in Eq.~\eqref{eq:main}, that expansion containing the smaller Wronskian determinant, of size $2p\times 2p$, as the upper\/-\/left minor but itself appearing near~$\consta(p)$\,?
\end{open}

Let us keep in mind, however, that for practical applications in the models of Theoretical Physics (cf.~\cite{LadaStasheff,Stasheff2019Review}) it is the \emph{small} indices $p\sim 10^{\leqslant 1}$ for which there could be any experimental detection and measurement of the intensity and contribution (if any) of the collective, $N$-ary effects (at $N=2p$) beyond the pairwise, binary interactions (at $N=2$ for $p=1$) through Lie commutators~$[{\cdot},{\cdot}]$.

%\begin{problem}
%Find many exact values of the coefficients in the asymptotic growth law:
%$\alpha = 2$ is confirmed; $\beta = \,?$ (we have a good estimate);
%$\gamma = \,?$, $\delta = \,?$, $\varepsilon = \,?$, $\mu = \,?$; the tail
%in negative powers of $p = \,?$
%\end{problem}

%\begin{problem}
%The continuous limit in $p$: an ODE upon $\consta(p)$, for
%$p \in \mathbb{Q}$, $p \in \mathbb{R}$.
%\end{problem}

%\medskip
%\noindent\textbf{Perspectives.} 
\begin{open}\label{OpnAnyN}
The arity of Wronskians in Eq.~\eqref{eq:main} is even, $N=2p$, 
yet the Wronskian determinants --\,viewed as multi\/-\/linear totally antisymmetric brackets on the algebra of functions $C^\infty(\BBR^1)$\,-- are known
to satisfy all the strongly homotopy Lie\/-\/algebra's $(\ell,m)$-\/Jacobi identities for arbitrary arities $\ell,m\geqslant 2$; a proof is found in~\cite{DzhumaWron} (see also~\cite{Kiselev2005AssociativeWronskians}). %[ForKac]
The combinatorics that works here is far more robust than the homotopy Lie theory in the associative algebra case of~\cite{HanlonWachs,DzhumaAssoc} for purely even arities~$\ell,m$.

Can the construction of alternating compositions in Eq.~\eqref{eq:opdef} and the count of Wronskians in Eq.~\eqref{eq:main} be extended to make $\consta(N/2)$ well defined for \emph{all} natural numbers $N\geqslant 2$, including odd integers that yield the differential orders $p\in\tfrac{1}{2}\BBZ$\,?\footnote{\label{FootHalfIntOrder}
To incorporate the orders $p\in\tfrac{1}{2}\BBZ$, 
one may prefer the circle $\mathbb{S}^1$ instead of the line $\BBR^1$ 
as the domain of coefficients for vector fields from $\Vect(\BBS^1)$ 
rather than $\Vect(\BBR^1)$.}

Likewise, we remember that the (in)complete generalised Wronskian determinants $W_{d\geqslant 1}^{k\geqslant 1}$ of differential orders~$k$ over $d$-dimensional base do satisfy all the strongly homotopy Lie\/-\/algebra's Jacobi identities
(\cite{Kiselev2005AssociativeWronskians,Kiselev2025Multidim}).
Can the generalised Wronskians% $W_{d\geqslant 1}^{k\geqslant 1}$
, seen as $N$-ary brackets, be realised in terms of alternated differential operators\,? %And 
If so, what will be the new combinatorial objects $\consta\bigl( (p_1,\ldots,p_d)
\bigr)$ that count the Wronskians over a $d$-\/dimensional base --- with $p_i$ denoting the respective differential order w.r.t.\ the independent variables $x_1$,\ $\ldots$,\ $x_d$\,?\\[0.5pt]
\centerline{\rule{1in}{0.7pt}}
\end{open}
   %Generalise for $p \in \mathbb{Z}$ odd; the
   %multi-dimensional generalisation over
   %$\mathbf{x} = (x^1, \ldots, x^d) \in \mathbb{R}^d$.

%\subsubsection*
\noindent\textbf{\textit{{Approbation.}}}%\label{sec:approbation}
\quad
This work has been reported by the first author at the 17th Algorithmic
Number Theory Symposium -- ANTS17 (6--10 July 2026 in Groningen, NL: poster
session), at the 36th International colloquium on group theoretic methods in Physics -- \textsc{Group36} (13--17 July 2026 in Valladolid, Spain), and at the Geometry and Quantum Theory -- GQT cluster conference (24--28 August 2026 in Malden near
Nijmegen, NL).
\quad
The second author spoke at the Prague Mathematical Physics seminar in the
Charles University (on 7 May 2026 in Prague, Czech Republic), at the
Mathematics seminar in the IH\'ES (13 May 2026 in Bures\/-\/sur\/-\/Yvette,
France), and at the Open Problems session within the 2nd ISNMP conference
on Nonlinear Mathematical Physics (28 June -- 4 July 2026 in Bad Ems,
Germany).
\quad
The authors are grateful to the organizers of all these events for
stimulating discussions and advice.

\medskip%\subsubsection*
\noindent\textbf{\textit{{Acknowledgments.}}}%\label{sec:ack}
\qquad
The first author thanks R.~Buring and G.~Petrella for discussions and
advice how to speed up the computation yet reduce the memory allocation;
suggested improvements made feasible the computation of
$\consta(p \geqslant 15)$. R.~Buring amended the code and independently
verified the values for $p \leqslant 7$.
The first author thanks the Center for Information Technology of the
University of Groningen for access to the High Performance Computing
cluster, H\'abr\'ok, and K.~C.\ Waggeveld for performing the computations
thereon. The travel of the first author was %partially 
supported by the Bernoulli Institute via projects `Algebra' and `MSc programme'.
\quad
The second author thanks M.~Kontsevich and O.~Zaboronski for discussions,
stimulating remarks, and insights. The travel of the second author was partially
supported by the Charles University, %in Prague
by the $\smash{\text{IH\'ES}}$, and by the Bernoulli Institute 
via project `Algebra'; %110135.
the stay at the ETH Zurich is supported in part by the Bernoulli Swiss scientific travel fel\-low\-ship (Switzerland).
  %Reims.

\section*{Declarations and statements}\label{sec:declarations}

\noindent%
\textbf{Authors' contribution statement.}\quad
Both authors contributed equally to this work.
   %G.Petrella: p=16,17;

\smallskip\noindent%
\textbf{Funding statement.}\quad%\subsection*{Funding}
No funding for this %the submitted 
work was obtained from any organisation.

\smallskip\noindent%
\textbf{Conflicts of interest statement.}\quad%\subsection*{Conflicts of interests}
On behalf of both authors, the corresponding author states that there is no conflict of interest.

\smallskip\noindent%
\textbf{Data availability statement.}\quad
The values of $\consta(p)$ at $1\leqslant p \leqslant 18$ 
are contained in Appendix~\ref{app:data} from p.~\pageref{app:data} onwards in this paper.
The new integer sequence $\{ \consta(p) \}_{p\in\BBN_{\geqslant 1}}$ itself is registered as %for the new integer sequence 
OEIS A392\,714 in the On-Line Encyclopedia of Integer Sequences (\cite{oeisA392714}).
The program code (spread into several files in \textsf{Python} and \textsf{C${+}{+}$}) to compute $\consta(p)$ is available from~\cite{ShahCode2026} on \textsf{github}.
\quad
The new 206-\/decimal\/-\/digit prime number from the factorisation of $\consta(p=14)$ is now recorded in the prime factor database \textsf{factordb.com}.
The certificates of primality for the 31-, 46-, 52-, 85-, and 206-\/decimal\/-\/digit factors of the integers $\consta(9)$,\ $\consta(10)$,\ $\consta(11)$,\ $\consta(12)$,\ and $\consta(14)$, respectively, are posted on the \textsf{github} page~\cite{ShahCode2026} of the first author.

\smallskip\noindent%
\textbf{Use of AI statement.}\quad%\subsection*{Use of AI}
The original, non-optimal program code was written in full by the first author in \textsf{Python} (realising Algorithm~\ref{alg:old}: see Appendix~\ref{app:oldcode}). During the
GQT cluster conference (Malden, 24--28 August 2026), where this work was presented, G.\,Petrella translated the program code into \textsf{Julia}~\cite{PetrellaJulia}, which made it roughly ten times faster. The \textsf{Python} and \textsf{Julia} programs were then given to the AI model \textsf{Claude (Anthropic)} with the request to make the computation as fast and as memory\/-\/efficient as possible. The idea, then received from \textsf{Claude}, was to use a recursive algorithm in which all arithmetic is done modulo primes, with $\consta(p)$ reconstructed by the Chinese Remainder Theorem (this is now Algorithm~\ref{alg:deployed}).\footnote{\label{FootModManyPrimes}
The collegial advice to calculate faster by doing so modulo (big enough) primes, and to reconstruct the integers $\consta(p)>0$ by using the Chinese Remainder Theorem, had been received by the authors during the 17th Algorithmic Number Theory Symposium -- ANTS17 (held on 6--10 July 2026)%, i.e.\ months before the reiteration of that same advice by \textsf{Claude}
.} 
%%%
To implement this, \textsf{Claude} split the code into two languages: 
\textsf{C${+}{+}$} solvers which evaluate the recursion modulo each prime, and a \textsf{Python} driver which chooses the primes and combines the results.\footnote{\label{FootAIonMemoryUse}
For example, at $p=15$ the economy of operative memory usage was from 60 GB down to 1.16 GB, compared with the runs of the old program code in Appendix~\ref{app:oldcode}.}
%%%
The re\/-\/calculated values $\consta(p)$ agree with those of the original
\textsf{Python} program wherever both were run ($p\leqslant 14$).

%\bibitem{Zaboronski}%{FootZabor} on p.--"--
%O. Zaboronski, Private communication, Note of 22 June 2026 and subsequent up-
%dates, 2026.

\newpage
\appendix
\section{Proof of Theorem~\ref{thm:main}}
\label{app:proof}
\noindent%
The proof relies on the following two lemmas and the corollary they imply.

\begin{lemma}[Antisymmetry]\label{lem:antisym}
The alternating composition satisfies, for $a \neq b$,
\[
  \mathcal{A}_p[\ldots, w_a, \ldots, w_b, \ldots](f)
  \;=\;
  -\,\mathcal{A}_p[\ldots, w_b, \ldots, w_a, \ldots](f).
\]
\end{lemma}

\begin{proof}
Compose each summation index $\sigma$ in \eqref{eq:opdef} with the
transposition $\tau = (a\ b)$. As $\sigma$ ranges over all of $S_N$, so
does $\sigma\circ\tau$, and $(-1)^{\sigma\circ\tau} = -(-1)^\sigma$.
Simultaneously, composing with $\tau$ interchanges the labels $a$ and $b$
throughout, which is equivalent to swapping $w_a$ and $w_b$ in the weight
list. Relabeling $\sigma \mapsto \sigma\circ\tau$ in the sum thus changes
its sign, proving the claim.
\end{proof}

\begin{lemma}[Derivative count]\label{lem:count}
In every monomial produced by the full Leibniz expansion of any summand of
$\mathcal{A}_p[w_1,\ldots,w_N](f)$, the total number of $x$-derivatives is
exactly
\[
  N \cdot p \;=\; 2p^2.
\]
\end{lemma}

\begin{proof}
Each summand in \eqref{eq:opdef} is the composition of $N = 2p$ operators,
each contributing exactly $p$ applications of $\partial_x$. The Leibniz
rule redistributes these derivatives among the factors in every term, but
preserves the total count. Thus every monomial in the expansion carries
exactly $Np = 2p^2$ derivatives.
\end{proof}

\begin{cor}[Order of $\mathcal{A}_p$]\label{cor:orderp}
The operator $\mathcal{A}_p[w_1,\ldots,w_N]$ has differential order
exactly $p$ in $f$. That is, every surviving term in the expansion of
\eqref{eq:opdef} carries the factor $\partial_x^p(f)$, and no term
$\partial_x^k(f)$ with $k \neq p$ appears with nonzero coefficient.
\end{cor}

\begin{proof}
Let $k$ denote the differential order of $f$ in a given term of the full
Leibniz expansion, so that the term contains the factor
$\partial_x^k(f)$.

\smallskip
\noindent\emph{Lower bound $k \geqslant p$.} In each summand of
\eqref{eq:opdef}, the innermost operation is
$w_{\sigma(N)}\partial_x^p(f)$, which already contains $\partial_x^p(f)$.
Every subsequent application of $\partial_x^p$ from the left acts on an
expression involving $\partial_x^p(f)$ via the Leibniz rule: for smooth
$g$ and $m \geqslant 0$,
\[
  \partial_x^m\!\left(g\cdot\partial_x^p(f)\right)
  = \sum_{l=0}^{m}\binom{m}{l}\partial_x^l(g)\cdot\partial_x^{p+m-l}(f),
\]
so the order of $f$ in each resulting term is $p + m - l \geqslant p$.
Proceeding inductively through all $N$ composition steps, every monomial
satisfies $k \geqslant p$.

\smallskip
\noindent\emph{Upper bound $k \leqslant p$.} We can write the $2p^2$
derivatives of Lemma~\ref{lem:count} as $k$ acting on $f$ and $2p^2 - k$
acting on the weight functions, with nonnegative orders
$k_1, \ldots, k_N$ (that is, $k_i$ derivatives on $w_i$) satisfying
$\sum_{i=1}^N k_i = 2p^2 - k$.

By Lemma~\ref{lem:antisym}, if $k_a = k_b$ for some $a \neq b$, then
swapping $w_a$ and $w_b$ maps the monomial to itself while negating its
coefficient, forcing that coefficient to zero. Thus every term with
nonzero coefficient has \emph{pairwise distinct} values
$k_1, \ldots, k_N$.

Since $k_1, \ldots, k_N$ are $N = 2p$ pairwise distinct nonnegative
integers,
\begin{equation}\label{eq:lb}
  \sum_{i=1}^{N} k_i
  \;\geqslant\;
  0 + 1 + 2 + \cdots + (N-1)
  \;=\;
  \frac{N(N-1)}{2}
  \;=\;
  p(2p-1).
\end{equation}
Thus
\[
  2p^2 - k \;=\; \sum_{i=1}^{N}k_i \;\geqslant\; p(2p-1) = 2p^2 - p,
\]
which gives $k \leqslant p$. Combining the two bounds yields $k = p$.
\end{proof}

\begin{rem}\label{rem:budget}
The lower bound in \eqref{eq:lb} is tight: equality holds if and only if
$\{k_1, \ldots, k_N\} = \{0, 1, \ldots, N-1\}$. The budget for weight
derivatives is exactly $2p^2 - p = p(2p-1)$, matching this minimum
precisely. Every other profile of pairwise distinct nonnegative integers
--- for instance $\{0,1,2,3,4,6\}$ when $N = 6$, $p = 3$, requiring
$0+1+2+3+4+6 = 16 > 3 \cdot 5 = 15$ derivatives --- exceeds the available
budget and cannot contribute.
\end{rem}

\begin{proof}[Proof of Theorem~\ref{thm:main}]
By Corollary~\ref{cor:orderp}, every surviving term in the expansion of
the operator $\mathcal{A}_p[w_1,\ldots,w_N](f)$ has the form
\[
  C_{\mathbf{k}}\cdot\prod_{i=1}^{N}\partial_x^{k_i}(w_i)\cdot
  \partial_x^p(f),
\]
where $C_{\mathbf{k}}$ is a numerical coefficient and
$\mathbf{k} = (k_1, \ldots, k_N)$ is the order profile on the weights.

As shown in the proof of Corollary~\ref{cor:orderp}, antisymmetry forces
each surviving profile to have pairwise distinct $k_i$. With $k = p$
fixed, the derivative budget on the weights is $2p^2 - p = p(2p-1)$, which
by Remark~\ref{rem:budget} is exactly the minimum achievable by $N$
pairwise distinct nonnegative integers. Thus equality must hold in
\eqref{eq:lb}, and the \emph{unique} surviving profile is
\[
  \{k_1, \ldots, k_N\} \;=\; \{0, 1, \ldots, N-1\}.
\]
The alternating sum over all assignments of this profile to the weight
functions is, by definition, the Wronskian:
\[
  \Wronsk(w_1,\ldots,w_N)
  \;:=\;
  \det\!\bigl[\partial_x^{j-1}(w_i)\bigr]_{1\leq i,j\leq N}
  \;=\;
  \sum_{\sigma\in S_N}(-1)^\sigma\prod_{i=1}^{N}
  \partial_x^{\sigma(i)-1}(w_i).
\]
The coefficient $C_{\mathbf{k}}$ arises purely from the combinatorics of
the iterated Leibniz expansion and depends only on $p$, not on $f$ or on
the weight functions. Denoting this factor by $\consta(p)$, we
conclude
\[
  \mathcal{A}_p[w_1,\ldots,w_N](f)
  \;=\;
  \consta(p)\cdot\Wronsk(w_1,\ldots,w_N)\cdot\partial_x^p(f).
  \qedhere
\]
\end{proof}

\section{The old, non-optimal code}
\label{app:oldcode}
\subsection{Maple}\label{app:maplecode}

{\footnotesize
\begin{verbatim}
restart;
with(combinat): with(GroupTheory): with(LinearAlgebra):
constp := proc(p)
    local N, ComposeOp, expression, sigma, M, Wr, RHS;
    N := 2*p;
    ComposeOp := proc(sigma)
        local expr, j;
        expr := f(x);
        for j from N by -1 to 1 do
            expr := cat(w, sigma[j])(x)*diff(expr, [x $ p]);
        end do;
        expand(expr);
    end proc;
    expression := 0;
    for sigma in permute(N) do
        expression := expression
            + PermParity(Perm(sigma))*ComposeOp(sigma);
    end do;
    expression := expand(expression);
    M := Matrix(N, N, (i, j) -> diff(cat(w, j)(x), [x $ (i - 1)]));
    Wr := Determinant(M);
    RHS := expand(Wr*diff(f(x), [x $ p]));
    simplify(expression/RHS);
end proc:
for p from 1 to 3 do print(p, constp(p)) end do:
\end{verbatim}
}

\subsection{Python}
\noindent%
For convenience, we keep here the old, short code, which refers to the
initial memoize backtracking approach. Run as printed, it reproduces
$\consta(p)$ for all $p \leqslant 7$ within seconds. The complete package, including the layered multi-modular implementation of
Algorithm~\ref{alg:deployed}, is available at \cite{ShahCode2026}.

{\footnotesize%
\begin{verbatim}
from math import prod, factorial
from functools import lru_cache

def a(n):
    N = 2*n
    @lru_cache(maxsize=None)
    def b(avail, s, d):
        if d == N-1: return 1
        t = 0
        for i, e in enumerate(avail):
            ns = s + e - n - 1
            if ns < 0: continue
            sgn = -1 if (len(avail)-1-i) & 1 else 1
            t += (sgn * prod(range(ns+1, ns+n+1))
                  * b(avail[:i]+avail[i+1:], ns, d+1))
        return t
    return (abs(b(tuple(range(2, N+1)), 0, 0))
            // prod(factorial(k) for k in range(N)))

print([a(n) for n in range(1, 8)])
\end{verbatim}

}%\footnotesize

\bigskip
\noindent$\boldsymbol{\copyright}$\qquad%\subsubsection*{Copyright notice}
The copyright for all newly designed program files %software modules 
which are %have been 
referred to in %presented in 
this paper is retained by K.\,C.\,Shah;
provisions of the~MIT free software license apply.
   %the update which allowed computing c(p=15) was suggested by dr. G.Petrella;

\section{Values of \texorpdfstring{$\consta(p)$}{const(p)} up to
\texorpdfstring{$p = 18$}{p = 18}, and a new 206-decimal-digit prime}
\label{app:data}
\noindent%
Here we present the known values of $\consta(p)$ up to $p = 18$,
together with their prime factorizations.

%The order of battle.
The values $\consta(p=3)=90$ (see Table~\ref{tab:intro}), $\consta(p=4)$,\ $\ldots$,\ $\consta(p=15)$, and $\consta(p=18)$ were first obtained by the first author: on a laptop for $p\leqslant 15$ and on H\'abr\'ok high performance computing cluster (CIT, University of Groningen) for $p=18$; here, the first author enjoyed advice and help from K.\,C.\,Waggeveld.
The values $\consta(p=16)$ and $\consta(p=17)$ were first obtained and posted on the \textsf{github} page by dr.\ G.\,Petrella (between 25 August and 1 September 2026); we re\/-\/computed (on H\'abr\'ok) --\,and we now confirm\,-- these two values.\footnote{\label{FootRedo}
Quite naturally, every modification or acceleration of the algorithm, which we describe in section~\ref{sec:algorithms} and in~\cite{SciPost}, was immediately tested for reproduction of all the previously known values $\consta(p)$ at smaller indices~$p\geqslant 1$. In May 2026, then\/-\/known values $\consta(p\leqslant 7)$ were verified independently by R.\,Buring at Ecole Polytechnique (France) by using program code in C/FLINT (\textbf{F}ast \textbf{Li}brary for %\textbf{I}nteger 
\textbf{N}umber \textbf{T}heory, code: 3.5.0 release).}%RB:laptop;

The prime factorizations below were obtained with Alpern's integer
factorization calculator \cite{alpertron}, which in our usage reliably
extracts prime factors up to about $30$ decimal digits. For each
$\consta(p)$, once no further factor below this threshold could be
found, the remaining cofactor was not factored further; instead we
subjected it to a primality test in PARI/GP \cite{PARI}. A cofactor that passes
this test is reported below as a single unfactored integer, and is proven
prime; for each of the nine prime factors with at least $20$ digits, an ECPP
primality certificate \cite{AtkinMorain1993} is available in the folder
\texttt{certificates} of~\cite{ShahCode2026}. 
%%%
The exceptions --\,which failed the probabilistic PARI/GP test\,-- 
are the large cofactors appearing in the
factorizations of $\consta(13)$, $\consta(15)$, $\consta(16)$, $\consta(17)$ and $\consta(18)$; %, which fail the test and 
these exceptions are %accordingly 
labelled ``(Composite)'' without being split further; 
on each of these five cofactors, the elliptic curve method \cite{Lenstra1987}, run with GMP-ECM \cite{GMPECM} to the $35$-digit level, found no factor (so, those must be very big).

\medskip

\constentry{1}{2}{1}{1}

\constentry{2}{4}{2}{2}

\constentry{3}{6}{90}{$2$ $\cdot$ $3^2$ $\cdot$ $5$}

\constentry{4}{8}{586656}{$2^5$ $\cdot$ $3^3$ $\cdot$ $7$ $\cdot$ $97$}

\constentry{5}{10}{1915103977500}{$2^2$ $\cdot$ $3$ $\cdot$ $5^4$ $\cdot$ $7$ $\cdot$ $79$ $\cdot$ $103$ $\cdot$ $4483$}

\constentry{6}{12}{7886133184567796056800}{$2^5$ $\cdot$ $3^6$ $\cdot$ $5^2$ $\cdot$ $11$ $\cdot$ $23$ $\cdot$ $223$ $\cdot$ $239$ $\cdot$ $1002820739$}

\constentry{7}{14}{85873408332103907284746052081828368}{$2^4$ $\cdot$ $3^2$ $\cdot$ $7^6$ $\cdot$ $13$ $\cdot$ $29$ $\cdot$ $1361$ $\cdot$ $2928279181$ $\cdot$ $3373618837829$}

\constentry{8}{16}{4594491123326092088701002220876785865521537220214784}{$2^{17}$ $\cdot$ $3^2$ $\cdot$ $7^5$ $\cdot$ $13$ $\cdot$ $269$ $\cdot$ $9661$ $\cdot$ $4772558055911$ $\cdot$ \seqsplit{1437228585673967071537}}

\constentry{9}{18}{2059560342372549855311520646192231857828658156995213197895311773186910208}{$2^{11}$ $\cdot$ $3^{14}$ $\cdot$ $7^2$ $\cdot$ $11^2$ $\cdot$ $13$ $\cdot$ $17$ $\cdot$ $43$ $\cdot$ $151$ $\cdot$ $2536393$ $\cdot$ $1193919540273283$ $\cdot$ \seqsplit{8160866422780307366231285008303} (31 digits, prime)}

\constentry{10}{20}{12366585616687922175229690518161006581890827598464384763435846689194555321926721635186969600000000}{$2^{16}$ $\cdot$ $3^6$ $\cdot$ $5^8$ $\cdot$ $19$ $\cdot$ $2318471913717203$ $\cdot$ \seqsplit{3489987569680376097511} $\cdot$ \seqsplit{4310255301386067349690402472139980447858870477} (46 digits, prime)}

\constentry{11}{22}{1512170566890529388470809077236613950878090470607103334738020704327297913201742894343506377843406342016226367608547680866000000}{$2^7$ $\cdot$ $3^5$ $\cdot$ $5^6$ $\cdot$ $11^{10}$ $\cdot$ $13$ $\cdot$ $19$ $\cdot$ $25457$ $\cdot$ $1304111$ $\cdot$ $19027007$ $\cdot$ $277381877879789$ $\cdot$ \seqsplit{2724104292276785198723} $\cdot$ \seqsplit{1017528619742366754559096610489525888417939019147331} (52 digits, prime)}

\constentry{12}{24}{5498204731735819824932764431577828853091535075963690476426743279778624624056195101600179165694626009665228817515586720358929229075113647839992833114429030400000}{$2^{18}$ $\cdot$ $3^{11}$ $\cdot$ $5^5$ $\cdot$ $11^8$ $\cdot$ $23$ $\cdot$ $6240374221$ $\cdot$ $4156745629145249267$ $\cdot$ \seqsplit{150087192783934271953579} $\cdot$ \seqsplit{1973877182896509567154686476569916604851423966617859790944299941672435839359560447739} (85 digits, prime)}

\constentry{13}{26}{839630392364690485033794364133039282555694250991715962816289209038314705544359344129440199511545622749161083016799037230605187931570735637773841615964174834585398693013887219584559285780585126502400}{$2^{13}$ $\cdot$ $3^8$ $\cdot$ $5^2$ $\cdot$ $11^6$ $\cdot$ $13^{12}$ $\cdot$ $17$ $\cdot$ $97$ $\cdot$ $337$ $\cdot$ \seqsplit{27243397385914871724447417691488851104574713227852664272084500718530055050039017404332321041867964292700649936028447331458925288320887246456354056141933129121535401} (164 digits) (Composite)}

\constentry{14}{28}{7398040169232585459350131400876616433108456563542326671820170467323437649015925253611993741824016395372908844522732645092417947941972247942351664804568139456872774721959807686464898490851546742468976729501727604943676101572611708283345305600}{$2^{17}$ $\cdot$ $3^3$ $\cdot$ $5^2$ $\cdot$ $7^{13}$ $\cdot$ $11^4$ $\cdot$ $13^{11}$ $\cdot$ \seqsplit{32891285673563672246003165608072510232318699774051357913081882353000559529851490628731802683023201810328925871944565109492955446073147116264586137549690490817951980312345150293192647954551600837443915441259} (206 digits, \textbf{new prime!})\qquad
When discovered (by using Alpern's integer factorisation calculator~\cite{alpertron}), this 206-\/decimal\/-\/digit cofactor of $\consta(p=14)$ passed the probabilistic primality test in PARI/GP~\cite{PARI}. The new number was then \emph{proven} prime: its ECPP primality certificate is posted (in the folder \texttt{certificates}) ot the page~\cite{ShahCode2026}. This new big prime is now known in the factor database \textsf{factordb.com}.}

%\bigskip
\constentry{15}{30}{5046280762981107750679104744682397945183644990510297983747447227739387378343262856276570057278000846892728778975597920559603289367706909388720810389015995734049432105360766579402308571357991670074027401802837012706827921637362408233152407377844612325186847659973176040350451124500000000000}{$2^{11}$ $\cdot$ $3^{12}$ $\cdot$ $5^{12}$ $\cdot$ $7^{8}$ $\cdot$ $11^{3}$ $\cdot$ $13^{8}$ $\cdot$ $17$ $\cdot$ $29$ $\cdot$ $1123$ $\cdot$ $17932261$ $\cdot$ $8489704739$ $\cdot$ \seqsplit{35998370445164576398101362857001143665554257672255123083172710733119209893273858946239169040707559122646097087923030609777705717637816136357406155850135970673341568179916114837126949781310546805693740854278750149253153844930248619} (230 digits) (Composite)}

\constentry{16}{32}{350347068714984736363932448237551970947101428496433603230732475590577900010181317126682573243908900673762047143180626209387668068243491830426020243032372403718844363541878051373670782851465908976221799839150582908549751779698942620965693605901539896056419209507463625246596806195287565518966511086429743510950560190101128006162062508032000000}{$2^{49}$ $\cdot$ $3^{3}$ $\cdot$ $5^{6}$ $\cdot$ $7^{4}$ $\cdot$ $11$ $\cdot$ $13^{6}$ $\cdot$ $31$ $\cdot$ $53$ $\cdot$ $13996862959$ $\cdot$ $1465248299707$ $\cdot$ $17702417285659$ $\cdot$ \seqsplit{19399381221922156394859726481744458639754617182495325977416723636130785464893244452976916241528983739855884897178054454260725970969267639142211137008175006605365104441340775719092268370556183219090726164967115024280373741600899730002362383006544058860601785950398780953433} (272 digits) (Composite)}

\constentry{17}{34}{3197965883718109658323875604160413673114110008256001027316537941562490229065710229228394006993322313820835800118977741980981666285691001688471038276184835301225716345256441297648297778156138839796695800796167935359107753061518287142550811266192801172345646571321353980199954353184664914329752958720466214691618390475966876541322445714517556045466760936086521525116490257124597804914826798986071900160}{$2^{33}$ $\cdot$ $3^{3}$ $\cdot$ $5$ $\cdot$ $7^{2}$ $\cdot$ $13^{5}$ $\cdot$ $17^{16}$ $\cdot$ $19$ $\cdot$ $31$ $\cdot$ $877$ $\cdot$ $929$ $\cdot$ $3257$ $\cdot$ $2689989762199$ $\cdot$ \seqsplit{740893235261480061066369800736286834375200993794572718877900823673459171856568150710971203521710470310935010864805341984387145148222914842648798528800359668217848827563577111175460530206534928328665520912912418440566330019075622068279081603539327836363280503601529048141594582483643558675360408890717800316663382909629248428054934137259} (336 digits) (Composite)}

\constentry{18}{36}{4881345103627038434382130908280685053592076051971210164732335609526015977369348085157295992786258699940148334639842817509456108934457975820613023024887154481856472085069145948635571059231120856851092723894345630267577594029979505586755230940245438340156251154793718275320412012578420220836631816065580811293601113213338449375184941564894156534923570432671711739323358307397134414181151074757894647732483849095839327364106371485172952704772018032231441635265740800}{$2^{32}$ $\cdot$ $3^{31}$ $\cdot$ $5^{2}$ $\cdot$ $13^{2}$ $\cdot$ $17^{15}$ $\cdot$ $134011379699$ $\cdot$ \seqsplit{1135320664536347698997355677019288226996440304812297179826252019490655673960903598904434917814856583463331267604665524404316235193491942218352826423172687088348533326208332980465822220210215338297352387643698008355949353731654718511703019073944146898822864417030951270778949149503295129620590760576030600991879865152205018289039791250429542395467439020268700915048055047582500470106401165398601545848352717} (406 digits) (Composite)}

\end{document}